\documentclass[]{interact}

\usepackage{epstopdf}
\usepackage[caption=false]{subfig}

\usepackage[numbers,sort&compress]{natbib}
\bibpunct[, ]{[}{]}{,}{n}{,}{,}
\renewcommand\bibfont{\fontsize{10}{12}\selectfont}
\makeatletter
\def\NAT@def@citea{\def\@citea{\NAT@separator}}
\makeatother

\newtheorem{thm}{Theorem}

\newtheorem{cor}{Corollary}
\newtheorem{lem}{Lemma}
\newtheorem{ass}{Assumption}
\renewcommand{\P}{\mathbb{P}}
\newcommand{\E}{\mathbb{E}}

\begin{document}

\title{\bf Sampling distributions and estimation for multi-type Branching Processes.}
\author{Gonzalo Contador, Departamento de Matem\'atica, Universidad T\'ecnica Federico Santa Mar\'ia
\and
Bret M. Hanlon, Department of Biostatistics, University of Wisconsin-Madison.
}
\date{}
\maketitle
 
 \begin{abstract}
	This study focuses on a multi-dimensional supercritical branching process with offspring distribution in a parametric family. Each vector coordinate represents the number of offspring of a given type, and the process is observed under family-size sampling: a random sample is drawn from the population, and each individual reports its vector of brood sizes. We show that the probability of no siblings being sampled (so that the sample can be considered independent) converges to one under specific conditions on the sample size. Moreover, the sampling distribution of the observations converges to a weighted mixture of offspring distributions, enabling observation to be considered an i.i.d. sample of a mixture law for which standard inference methodology applies. We provide asymptotic distributions for the resulting estimators and conduct a simulation study using respondent driven sampling to assess their performance. 
.
 \end{abstract}

\newpage

\section{Introduction}\label{sec:Intro}

\indent Consider a multi-type (multidimensional) supercritical branching process whose offspring distributions (for each type of individual) belong to a parametric family. The process is observed under family-size sampling: a random sample of
individuals is drawn and each reports his brood size as a vector containing the number of siblings of each type. When this vector is one dimensional, under a bounded growth condition on the sample size, the sample as can be considered an independent and identically distributed sample of the so called \textit{size biased distribution} \cite{mm}, for which abundant statistical estimation and inference methods are available. In a multi-dimensional setting, the brood sizes in each coordinate are grouped according to a type or classification of individuals (presence of a genetic marker, for instance) and, as a key feature, each type of individual produces offspring according to a different probability distribution, so the multi-dimensional problem is characterized by as many probability distributions as types of individuals coexist in the population. Our contribution is to extend these results to such multi dimensional setting: we will provide sufficient conditions on the sample size and the different probability distributions of the problem so that the sample behaves asymptotically as an i.i.d. sample, and further show that this asymptotic identical distribution is simply a weighted average of the \textit{size biased} distributions of each type's offspring distribution. 

In \cite{guttorp1989}, a reference for the main results for statistical inference for branching processes is provided. As Galton-Watson branching processes (GWP) are characterized by their offspring distribution, statistical problems typically involve its estimation. The most commonly studied observation scheme for  a branching process is observing aggregated information of generation sizes (\cite{crump1972nonparametric, heyde1974estimating, dion1978statistical, lockhart1982non, guttorp1989, wei1990est, dion1994, kuelbs2011weak}). A second observation process is based on family-size sampling, which inherently carries a bias \cite{scheaffer1972size}. For example, if the first generation has two nodes, the first produces one child and the second produces two children, then the sampled brood sizes of the subsequent generation will report a 2 twice as often as a 1. \cite{mm} discusses parametric estimators of the offspring distribution mean that correct for the size bias inherent to this sampling scheme.

In the multi-dimensional setting where individuals of different type coexist in families among the population, aside from the previously described size bias and lack of independence when sampling individuals of the same family, a third issue arises as different type of individuals reproduce according to different probability distributions, which demands for more specificity when imposing non-extintion conditions and upon determining a unique sampling distributions. We now provide the mathematical notation used throughout this work.

Throughout this work, we distinguish vectors $\bm{x} \in \mathbb{R}^l$ and scalars $x \in \mathbb{R}$ according to their boldface notation. For a vector $\bm{W} \in \mathbb{R}^l$, we denote its $L_1$ norm by $|\bm{W}|=\sum_{i=1}^l |W_i|,$ where for a scalar $x$, $|x|$ denotes its absolute value. Let $\mathbb{N}$ be the set of all nonnegative integers. For $l \in \mathbb{N}, l \geq 2$, consider an $l$-type branching process $\{\bm{Z_{n}} \}_{n \in \mathbb{N}}$ defined by standard recursion, this is, if  $\bm{Z_{n-1}} =(Z_{n-1,1}, \ldots Z_{n-1,l})$ is the vector of population sizes at generation $n-1$, with $\bm{Z_0}$ a previously defined vector with $|\bm{Z_0}|=1$, with updates defined for $n\in \mathbb{N}$ as \begin{equation}
	\label{eq:rec}\bm{Z_{n}} = \sum^l_{j=1}\sum^{Z_{n-1,j}}_{i=1} \bm{Y^{(j)}_{n-1,i}},
\end{equation} where $\bm{Y^{(j)}_{n-1,i}}$ is the offspring vector of the $i$-th individual of type $j$ in generation $n-1$. For $j \in \{1, \ldots ,l\}$, the vectors $\bm{Y^{(j)}}_{\cdot, \cdot}$ are distributed according to $p_j$ and if for some $i$, $Z_{n-1,i}=0$ the corresponding sum is interpreted as the zero vector.  We also denote the total number of individuals with parent of type $i$ in generation $n$ as

\begin{equation}
	\label{eq:typeofind} S_{n-1,i}=\sum_{k=1}^{Z_{n-1,i}} |\bm{Y^{(i)}_{n-1,k}}|.
\end{equation} 

The probability mass functions $\{p_j\}_{j=1, \ldots, l}$ in $\mathbb{N}^l$  verify the two following conditions

\begin{ass}\label{ass:supercriticality} For $\bm{0}$ the zero vector, $\{\bm{e_1}, \ldots , \bm{e_l} \}$ the canonical basis of $\mathbb{R}^l$ and any $j \in \{1, \ldots, l\}$:
	
	\begin{itemize}
		\item[i)] $p_j(\bm{0})=0.$ (almost surely, all individuals generate offspring)
		\item[ii)] $\sum_{k=1}^l p_j(\bm{\bm{e_k}})<1$. (the probability of multiple offspring is strictly positive)
	\end{itemize}
\end{ass}

Assumption \ref{ass:supercriticality} guarantees non-extinction and  supercriticality of the process, and in particular guarantees that no family size corresponds to the zero vector (as it can't be sampled). This is: $$\P(\bm{Z_{n}}=\bm{0} \mbox{ for some } n \in \mathbb{N})=0\quad \text{and} \quad \P(\exists K>0,  |\bm{Z_{n}}|\leq K \mbox{ for all } n \in \mathbb{N})=0.$$

We define the \textbf{reproduction matrix} $M$ of the branching process as the $l \times l$ matrix with row $i$ equal to the expected offspring vector of an individual of type $i$, this is $M_{ij}=\E(Y^{(i)})_j$. For the reproduction matrix, we assume:

\begin{ass}\label{ass:prandeig} 
	\begin{itemize}
		\item[i)] $M$ is positively regular. This is, $\exists n \in \mathbb{N}$ such that $\forall i,k \in \{1, \ldots, l\}$, $(M^n)_{ik}>0$.
		\item[ii)] The \textbf{largest eigenvalue} of $M$ and its \textbf{corresponding eigenvector}, denoted by $\rho$ and $\bm{b}$ for the rest of this work, verify $\rho >1$ and $\bm{b}\geq 0$.
	\end{itemize}
\end{ass}

The inequality $\bm{b}\geq 0$ should be understood component-wise. Positive regularity in particular ensures that all types of individuals appear infinitely often in the population. Assumption \ref{ass:prandeig} extends the non extinction result in \ref{ass:supercriticality} to each type \cite{posreg}, i.e. for each $j \in \{1, \ldots ,l\}$, $\P(\exists n_0,  |\bm{Z_{n,j}}|=0 \mbox{ for all } n \geq n_0)=0.$ The Perron-Frobenius theorem \cite{frobenius1912matrizen} guarantees the existence of an eigenvector with non-negative components whose the corresponding eigenvalue will be real, non-negative and its absolute value greater than or equal to all other eigenvalues \cite[Theorem XIII.3]{gantmakher2000theory}. Our assumption that $\rho>1$ is more restrictive and it is needed for Theorem \ref{thm:KST}, the Kesten-Stigum Theorem \cite{kesten1966}, which is a stronger version of the Martingale Convergence Theorem \cite{doob1961notes} for multi-dimensional branching processes \cite{ksarticle}:

\begin{thm}\label{thm:KST} 
	For a supercritical, positively regular branching process with reproduction matrix $M$ with largest eigenvalue $\rho >1$ and associated eigenvector $\bm{b}$ such that $|\bm{b}|=1$, there exists a one dimensional, almost surely positive random variable $W$ such that
	
	$$\bm{Z_{n}}\rho^{-n} \rightarrow \bm{b}W \mbox{ almost surely.}$$

\end{thm}

Theorem \ref{thm:KST} follows from assumptions \ref{ass:supercriticality} and \ref{ass:prandeig}, and is of crucial importance for this work. For a proof, see \cite{kurtz1997conceptual}.

The process in \eqref{eq:rec} is observed under family-size sampling: a random sample of $r_n$ individuals from generation $n$ is chosen (without replacement), each reporting their brood size by a vector of length $l$.  We will denote the sample as $\bm{X_{n,1}} ,\ldots ,\bm{X_{n,r_n}}$ or, when the generation $n$ is understood, by $\bm{X_{1}} ,\ldots , \bm{X_{r}}$. Each $\bm{X_{t}}$ corresponds to $ \bm{Y^{(i(t))}_{n-1,s(t)}}$, this is, the family sizes of sampled individual $t$ corresponds to the offspring of individual $s(t)$ of type $i(t)$ in the previous generation. This sampling scheme does not necessarily yield independent data, however, as members of the same family may be observed, thus classical methods to estimate offspring distribution parameters do not apply. 

In the one dimensional formulation of\cite{mm}, the non-sibling sets are introduced: 

\begin{equation}
	\label{eq:nonsibling}
	D_n=\{ \mbox{ all } r_n \mbox{ selected individuals belong to different families}. \}
\end{equation}

This is, a sample $S=\{\bm{X_{1}} ,\ldots , \bm{X_{r}}\} \subset D_n$ whenever $t \neq \tilde{t} \implies (s(\tilde{t}),i(\tilde{t})) \neq (s(t),i(t))$ (no two individuals come from the same parent). When the sample is contained in $D_n$, independence follows and likelihood functions can be written as product of individual probabilities.

Regarding the sampling size and the growth rate for the process, we assume the following:

\begin{ass}\label{ass:cond} 
	The three following conditions hold
	
	\begin{itemize}
		\item[i)] $r_n^2\rho^{-n} \rightarrow 0$.
		\item[ii)] There exists $C \in \mathbb{R}$ such that
		$\E(|Y^{(i)}|^2)\leq C \mbox{ for } i \in \{1, \ldots, l\}$.
		\item[iii)] There exists $K \in \mathbb{R}$ such that
		$\E(|Y^{(i)}|^{-1})\leq K \mbox{ for } i \in \{1, \ldots, l\}$.
	\end{itemize}
	
\end{ass}

The first assumption says, heuristically, that the sampling size cannot be too large in order to observe an independent sample. The second and third assumptions provide bounds useful to obtain estimates on deviations. 

The rest of this work is organized as follows: Section \ref{sec:prelim} provides preliminary results on sampling individuals belonging to the same family, the main results on the asymptotic sampling distribution and estimation are presented in Section \ref{sec:main}, Section \ref{sec:ex} presents a simulation study example using synthetic data from respondent driven sampling and Section \ref{sec:disc} discusses the results of this work. All mathematical proofs are contained in Appendix \ref{sec:proof}.

\section{Sampling siblings}\label{sec:prelim}

We prelude our main results with an analysis of the sampling distribution of two individuals. We argue both
that the contribution of sampled siblings to this probability is asymptotically negligible and that the scenario in which siblings are sampled itself is highly unlikely.

\subsection{Sampling two individuals}\label{sec:ss}

For vectors $\bm{u}, \bm{v} \in \mathbb{N}^l$ we are interested in the probability that, knowing the composition of generation $n$, two individuals in it report family sizes $\bm{u}$ and $\bm{v}$, respectively. This is, in $\P(X_1=\bm{u}, X_2=\bm{v}|\bm{Z_{n}})$ .

\begin{lem}\label{lem:sam2}
	
	The conditional probability of sampling family sizes $\bm{u},\bm{v} \in \mathbb{N}^l$ can be written as
	
	$$\P(X_1=\bm{u},X_2=\bm{v}|\bm{Z_{n}})=\sum_{i=1}^lp_i(\bm{u})\left[a_i(\bm{u},\bm{v})p_i(\bm{v})+2\sum_{j>i}b_i(\bm{u})b_j(\bm{v})p_j(\bm{v})+c_i(\bm{u},\bm{v})\right],$$
	where $$a_i(\bm{u},\bm{v})=\binom{Z_{n-1,i}}{2}\frac{S_{n-1,i}(S_{n-1,i}-1)|\bm{u}||\bm{v}|}{|\bm{Z_{n}}|(|\bm{Z_{n}}|-1)}\E\left[\binom{|\bm{v}|+|\bm{u}|+\sum_{j=3}^{Z_{n-1,i}}|Y_{n-1,j}^{(i)}|}{2}^{-1}\right],$$
	$$b_{i}(\bm{u})=\frac{Z_{n-1,i}S_{n-1,i}|\bm{u}|}{|\bm{Z_{n}}|}\E\left[\left(|\bm{u}|+\sum_{w=2}^{Z_{n-1,i}}|Y_{n-1,w}^{(i)}|\right)^{-1}\right],$$
	and 
	
	$$c_i(\bm{u},\bm{v})=\begin{cases}
	\binom{|\bm{u}|}{2} \frac{S_{n-1,i}Z_{n-1,i}}{|\bm{Z_{n}}|}\E\left(\binom{|\bm{u}|+\sum_{w=2}^{Z_{n-1,i}}|Y_{n-1,w}^{(i)}|}{2}^{-1}\right) & \bm{u}=\bm{v} >1\\
	 & \text{otherwise}
	\end{cases}.$$
\end{lem}

The terms in Lemma \ref{lem:sam2} can be understood, respectively, as: sampling two \textbf{different} families of the same type $i$ ($a_i$), sampling exactly one family of type $i$ ($b_i$) and sampling the same family twice ($c_i$), an event that can only happen if $\bm{u}=\bm{v}$ and $|\bm{u}|>1$. Observe that the event $\{ \bm{u}=\bm{v} \wedge |\bm{u}|>1\}$ in Lemma \ref{lem:sam2} contains but it is not equal to $D_n^c$ in \eqref{eq:nonsibling}, as it merely states that the two sampled individuals come from families of identical size, a sufficient but not necessary condition for siblings.

We consider a heuristic approach to obtain the asymptotics of $\P(X_1=\bm{u}, X_2=\bm{v}|\bm{Z_{n-1}})$ using Corollary \ref{cor:rates} (stated at the beginning of section \ref{sec:proof}): The terms $|\bm{u}|+\sum_{w=2}^{Z_{n-1,i}}|Y_{n-1,w}^{(i)}|$ and $|\bm{u}|+|\bm{v}|+\sum_{w=3}^{Z_{n-1,i}}|Y_{n-1,w}^{(i)}|$ behave asymptotically as $S_{n-1,i}=\sum_{w=1}^{Z_{n-1,i}}|Y_{n-1,w}^{(i)}|$ since

$$\frac{|\bm{u}|+|\bm{v}|+\sum_{w=3}^{Z_{n-1,i}}|Y_{n-1,w}^{(i)}|}{S_{n-1,i}}=1-\frac{|\bm{u}|+|\bm{v}|}{S_{n-1,i}}\approx 1,$$ and furthermore, it can be generalized that for a sample of size $r_n$ that grows not too fast (in a sense that will be made precise later), fixing the sampled family sizes to $\{u_1, \ldots, u_{r_n}\}$, one obtains 

$$\frac{\sum_{w=1}^{r_n}|u_w|+\sum_{w=r_n+1}^{Z_{n-1,i}}|Y_{n-1,w}^{(i)}|}{S_{n-1,i}}=1-\frac{\sum_{w=1}^{r_n}|u_w|}{S_{n-1,i}} \approx 1.$$ 

By the Kesten-Stigum Theorem (Theorem \ref{thm:KST}), and the Continuous Mapping Theorem \cite[Theorem 2.7]{billingsley1999convergence}, the term

$$\binom{Z_{n-1,i}}{2}\frac{S_{n-1,i}(S_{n-1,i}-1)}{|\bm{Z_{n}}|(|\bm{Z_{n}}|-1)}|\bm{u}||\bm{v}|\binom{|\bm{u}|+|\bm{v}|+\sum_{j=3}^{Z_{n-1,i}}|Y_{n-1,j}^{(i)}|}{2}^{-1},$$ which is the contribution by different sampled family sizes $|\bm{u}|$ and $|\bm{v}|$, behaves asymptotically as$|\bm{u}||\bm{v}|b_ib_j\rho^{-{2}}$ , the term

$$\frac{Z_{n-1,i}S_{n-1,i}Z_{n-1,j}S_{n-1,j}}{|\bm{Z_{n}}|^2}\biggl(\frac{|\bm{u}|}{|\bm{u}|+\sum_{w=2}^{Z_{n-1,i}}|Y_{n-1,w}^{(i)}|} \biggr)^2,$$
which is the contribution by different sampled families of equal size $|\bm{u}|$, behaves asymptotically as $|\bm{u}|^2b_i^2\rho^{-{2}}$, and lastly the term

$$\binom{|\bm{u}|}{2}\binom{|\bm{u}|+\sum_{w=2}^{Z_{n-1,i}}|Y_{n-1,w}^{(i)}|}{2}^{-1} \frac{Z_{n-1,i}S_{n-1,i}}{|\bm{Z_{n}}|},$$
the contribution by sampling the same family of size $|\bm{u}|$ twice, behaves as $\frac{b_i|\bm{u}|^2}{2\rho S_{n-1,i}} $, which vanishes almost surely as $n$ grows (regardless of the size of $\bm{u}$), thus suggesting that its contribution is negligible and that the probability of sampling families of two given sizes behaves  asymptotically as

$$\P(X_1=\bm{u},X_2=\bm{v})\rightarrow |\bm{u}||\bm{v}|\rho^{-{2}}\sum_{i=1}^l \sum_{j=1}^l p_i(\bm{u})p_j(\bm{v}) b_ib_j .$$

Formalizing this heuristic development would be tedious and notation heavy when sampling $r_n >>2$ individuals. To formally prove convergence of the sampling distribution for a sample of arbitrary size, we first show that, under mild assumptions on $r_n$, the absence of siblings in a sample is highly likely.

\subsection{The probability of selecting different families.}\label{sec:sdf}

The probability of selecting $r_n$ different families can be written as the expected value of a ratio between the total number of possible choices of $r_n$ individuals from  different families, $\sum_{s \in \mathbb{N}^l: |s|=r_n} \prod_{i=1}^{l}\binom{Z_{n-1,i}}{s_i}\prod_{j=1}^{s_i}|Y^{(i)}_j|$, and the total number of unrestricted choices of $r_n$ individuals from the overall population, $\binom{|\bm{Z_{n}}|}{r_n}$; this is, the probability of the non-sibling set $D_n$ in \eqref{eq:nonsibling} is

$$\P(D_n)=\E\Biggl(\biggl[\sum_{s \in \mathbb{N}^l: |s|=r_n} \prod_{i=1}^{l}\binom{Z_{n-1,i}}{s_i}\prod_{j=1}^{s_i}|Y^{(i)}_j|\biggr]\binom{|\bm{Z_{n}}|}{r_n}^{-1}\Biggr).$$

\begin{thm}\label{thm:assympDn}
	Under assumptions \ref{ass:supercriticality}, \ref{ass:prandeig} and \ref{ass:cond}, it holds that
	
	\begin{itemize}
		\item $\P(D_n)\rightarrow 1$ as $n\rightarrow \infty.$
		\item For $\alpha \in (0, -\log_\rho K)$, (with $K$ given by assumption \ref{ass:cond})
		$$\rho^{\alpha n}r_n^{-2}(1-\P(D_n))\rightarrow 0 \mbox{ as } n\rightarrow \infty.$$
	\end{itemize}
\end{thm}

The first result in Theorem \ref{thm:assympDn} supports dismissing the case of sampled siblings in the analysis as, in addition to having a negligible contribution in the sampling probability as argued in section \ref{sec:ss}, its probability $\P(D_n^c)$ shrinks to 0, with the second result in Theorem \ref{thm:assympDn} further assuring that this decay is faster than $K^nr_n^2$. 

\section{Asymptotic sampling distribution and estimation.}\label{sec:main}

\subsection{The limiting sampling distribution.}

The heuristic result developed in section \ref{sec:ss} suggests that, as $n \to \infty$, one observes

$$\P(X_1=\bm{u},X_2=\bm{v})\approx \frac{|\bm{u}||\bm{v}|}{\rho^2}\sum_{i=1}^l\sum_{j=1}^lp_i(\bm{u})p_j(\bm{v})b_ib_j.$$

The goal of this section is to generalize this result for the entire sample, provided $r_n$ is not too large as a function of $n$. A natural candidate for an asymptotic sampling distribution of $(X_1 ,\ldots , X_{r_n})$ is the distribution of $r_n$ independent and identically distributed random variables, each having probability mass function in $\mathbb{N}^l$ given by 

\begin{equation} \label{eq:sbias}
	p_S(u)=\frac{|\bm{u}|}{\rho}\sum_{i=1}^l p_i(\bm{u})b_i.
\end{equation}

Note that each of the summands in $p_S$ in \eqref{eq:sbias} resembles the size biased distribution proposed by \cite{maki1996role,patil1976size}, with  $p_S$ being a weighted average of the offspring distributions of the $l$ types by their (asymptotic) proportion given by the Kesten-Stigum theorem \ref{thm:KST}. Size-biased distributions are weighted in such a way that the probability of sampling an item is proportional to its size or magnitude \cite{mir2009pak} and therefore are ommonly encountered in probabilistic models for statistical analysis in lifetime \cite{al2019size}, forestry \cite{gove2003estimation}, genetics \cite{pennell2012trees}, among many other fields.

We prove convergence in law to the candidate distribution. We begin by establishing asymptotic convergence for a fixed set of observations:

\begin{thm}\label{thm:iid}
	Let $r>0$. Under assumptions \eqref{ass:supercriticality}, \eqref{ass:prandeig} and \eqref{ass:cond}, for a sample of size $r_n=r$ of family sizes in generation $n$, as $n \rightarrow \infty$

	$$\P(X_w=\bm{u}_w; w\in\{1, \ldots ,r_n\})\rightarrow \prod_{w=1}^{r_n}p_S(u_w).$$

\end{thm}

Theorem \ref{thm:asiid} formalizes convergence in distribution for $r_n=o(\rho^n)$ (i.e. verifying Assumption 3.i) via convergence of characteristic functions. The characteristic function of a random vector $Y$ is defined as the expected value of the complex exponential $\exp \{ -i\langle \bm{s},Y\rangle \}$ as a function of $\bm{s}$, where $\langle\cdot,\cdot\rangle$ is the usual inner product of $\mathbb{R}^l$, and $|\cdot|$ here is understood as the complex norm $|a+ib|^2=a^2+b^2$.

\begin{thm}\label{thm:asiid}
	Denote $\bm{t} =(\bm{t_1},\ldots , \bm{t_{r_n}} )$ with $\bm{t} \in \mathbb{R}^{l\times r_n}$  and all of the $\bm{t_i} \in \mathbb{R}^{l}$. Let $\phi_n$ be the characteristic function of a vector $\bm{X}=(X_1, \ldots , X_{r_n})$ of $r_n$ sampled family sizes. Let $\phi$ be the characteristic function of $r_n$ iid random vectors with distribution $p_S$ in \eqref{eq:sbias}. Under Assumptions \ref{ass:supercriticality}, \ref{ass:prandeig} and \ref{ass:cond}, as $n\rightarrow \infty$, uniformly for $\bm{t}$ in an open neighborhood of $\bm{0} \in \mathbb{R}^{l\times r_n}$ it holds that
	
	$$|\phi_n(\bm{t})-\phi(\bm{t})| \rightarrow 0.$$
	
\end{thm}

Theorem \ref{thm:asiid} guarantees that the sampling distribution of the family sizes converges to the distribution of as many independent random variables with distribution given by $p_S$ thanks to Levy's continuity theorem (see, for example, \cite{Williams1991ProbabilityWM}). It is therefore established that a sample of $r_n$ family sizes behaves asymptotically as an independent and identically distributed sample from the size biased distribution $p_S$ in \eqref{eq:sbias}.

\subsection{Asymptotic maximum likelihood estimation.}

Theorem \ref{thm:asiid} suggests that, for parametric estimation purposes, one could obtain maximum likelihood estimates by maximizing the asymptotic likelihood, which behaves like a product of marginal distributions in the limit. For $\{p_1, \ldots, p_l\}$ in a parametric family $\{p_1(\theta ), \ldots, p_l(\theta ): \theta \in \Theta \subset \mathbb{R}^s \}$ and a sample of family sizes $X_1, \ldots X_r$, maximum likelihood estimates can be obtained by maximizing 

$$\hat{\theta}_{ML}=\mbox{argmax}_{\theta \in \Theta} \prod_{j=1}^r p_S(X_j)=\mbox{argmax}_{\theta \in \Theta} \biggl[\rho^{-r} \prod_{j=1}^r |X_j|\sum_{i=1}^lb_ip_i(X_j)\biggr].$$

Imposing first order conditions for $\hat{\theta}_{ML}$, one has for $\theta=(\theta_1 ,\ldots , \theta_s)$ the set of $s$ equations given by

\begin{equation}\label{eq:amle}
	\sum_{j=1}^r \Biggl( \frac{\sum_{i=1}^l\frac{\partial p_i}{\partial \theta_d}(X_j)b_i+\frac{\partial b_i}{\partial \theta_d}p_i(X_j)}{\sum_{i=1}^l p_i(X_j)b_i}\Biggr)=\frac{r}{\rho}\frac{\partial \rho}{\partial \theta_d},  \quad d \in \{1,\ldots , s\}  .
\end{equation}

Several contributions have been made regarding estimation of $M$ and $\rho$ \cite{eigen1,chi2004,li2011note}. When the task at hand is parametric estimation for the offspring distributions of each the several types, explicit expressions for $\rho(\theta)$ and $\bm{b}(\theta)$ are required. When they are available, asymptotic normality is guaranteed under mild conditions, as per Theorem \ref{thm:amle}:

\begin{thm} \label{thm:amle}
	Denote by $\tilde{\theta}\in \mathbb{R}^s$ the true parameter of $p_1, \ldots, p_l$ from which $X_1, \ldots X_{r_n}$ are sampled. If $\rho(\theta)$, $\bm{b}(\theta)$ and $p_1(\theta ), \ldots, p_l(\theta)$ are all twice continuously differentiable in $\theta$ and the support of the offspring distributions $p_1, \ldots, p_l$ does not depend on $\theta$, as $r_n \to \infty$ for each $d \in \{1,\ldots , s\}$ 
	
	$$\sqrt{r_n}(\tilde{\theta}_d-(\hat{\theta}_{ML})_d) \rightarrow_\mathcal{L} \mathcal{N}\Biggl(0,\Biggl[\sum_{\bm{u}: p_S(\bm{u})>0}\Bigl(\frac{\partial p_S (\bm{u})}{\partial \theta_d}\Bigr)^2 \frac{1}{p_S (\bm{u})} \Biggr]^{-1} \Biggr).$$
\end{thm}

\subsection{Asymptotic method of moments estimates}

A second approach to parametric estimation is to provide method of moments estimates for  the largest eigenvalue and corresponding normalized eigenvector of the reproducing matrix $M$, built based on the result of theorem \ref{thm:asiid}. For a random vector $X$ in $\mathbb{N}^l$ with probability mass function $p_S$, one has
\begin{equation}
	\E(|X|^{-1})=\frac{1}{\rho} \quad \text{ and } \quad
	\E\Biggl(\frac{X_i}{|X|} \Biggr)=\frac{1}{\rho}\sum_{j=1}^lb_jM_{ji}=\frac{b_i\rho}{\rho}=b_i.
	\label{eq:mmvect}
\end{equation}

For a sample of family sizes, thanks to Theorem \ref{thm:asiid} and the Central Limit Theorem, the statistics

\begin{equation}\label{eq:mmest}
T_n=\frac{1}{r_n}\sum_{j=1}^{r_n}\frac{1}{|X_j|}\quad \text{ and } \quad \bm{U_n}=\frac{1}{r_n}\sum_{j=1}^{r_n}\frac{\bm{X_j}}{|X_j|},
\end{equation} are consistent and asymptotically normal estimators of $\rho$ and $\bm{b}$, respectively.

\begin{thm} \label{thm:moments} As $r_n \rightarrow \infty$,  $$\sqrt{r_n}(T_n^{-1}-\rho)\rightarrow_\mathcal{L} \mathcal{N}_1(0,\sigma_T^2\rho^4)\quad \text{ and } \quad\sqrt{r_n}(\bm{U_n}-b)\rightarrow_\mathcal{L} \mathcal{N}_l(\bm{0}, \Sigma).$$
	
	Where $\mathcal{L}$ denotes convergence in law, $\mathcal{N}_q$ is a $q$ dimensional multivariate normal,
	$$\sigma^2_T=\rho^{-1}\Biggl(\sum_{v \in \mathbb{N}^l}\sum_{k=1}^l \frac{b_kp_k(v)}{|\bm{v}|}-\rho^{-1} \Biggr) \text{, and } \Sigma_{i,j}=\rho^{-1}\sum_{v \in \mathbb{N}^l}\sum_{k=1}^l \frac{b_kp_k(v)v_iv_j}{|\bm{v}|} - b_ib_j.$$
\end{thm}

\section{Example: Respondent-Driven Sampling.}\label{sec:ex}

This section contains a simulation based illustration of the results in this paper.

	To effectively survey segments of the general population that are typically difficult to access through traditional sampling methods, a specialized sampling approach known as respondent-driven sampling (RDS) has been developed\cite{dh}. RDS  has  become  particularly  popular in HIV research as the populations most at risk for HIV (e.g., people who inject drugs, sex workers, and men who have sex with men) are hard to sample using conventional techniques \cite{roch2018generalized,gile2011improved,mccreesh2012evaluation}. Modelling the population as a connected graph, the RDS scheme consists on surveying members of the population and subsequently tasking them with handing out surveys in their network. As the HIV status of an individual depends on the (unknown) HIV status of the individuals in their network, the process can be modeled as a two type branching process, with $\bm{X}_j$ counting the numbers of HIV positive and negative individuals to whom individual $j$ handed out a survey. As a chain-referral method, RDS results are intrinsically dependent on the underlying network structure of the population \cite{sperandei2018assessing}, and estimates based on RDS data are known to be biased,  \cite{gile2011improved,goel2009respondent}, we will simulate network structures according to two different models: the Erdös-Renyi model \cite{erdHos1961strength}, connections between different individuals are generated independently with known probability, resulting in the same number of expected connections per individual and a more homogeneous network; and the Barabasi-Albert model \cite{barabasi1999emergence}, where individuals are added sequentially and each new individual chooses a fixed number connections at random, resulting in a \textit{rich get richer} type of network.
	
	We conduct a numerical study, proceeding in similar fashion than the authors in \cite{sperandei2018assessing} did:
	\begin{itemize}
		\item[$\bullet$] On a population of $N=10^4$ individuals, we simulate two different networks, one according to the Erdös-Renyi model \cite{erdHos1961strength}  with $10$ expected connections per individual and another according to the Barabasi-Albert model \cite{barabasi1999emergence}, with each new individual choosing 5 connections upon addition, also resulting in  $10$ expected connections per individual. These networks remain fixed throughout the study.
		\item[$\bullet$] In each network, we set transmission processes starting with $k=10,100,250$ randomly selected individuals to transmit the condition
		 to their contacts, with probability $0.05$ to each individual in their network. HIV is spread in the population step by step until a target proportion of at $5\%$ or $15\%$ is reached. 
		\item[$\bullet$] Starting from a randomly chosen individual, we start the RDS branching tree with $p_1$ and $p_2$ (the probability distribution of positive and negative HIV statuses) is as follows: each individual is tasked to hand out 1, 2 or 3 surveys with probabilities $0.18$, $0.18$ and $0.64$ independently among their network, so that the probability distributions for each vector depend on each individual's HIV status and are unknown to the researcher.
		\item[$\bullet$] As $\rho=2.46$, after $n=8$ generations surveyed, we choose a sample of $r_n=n^2=64$ surveyed individuals and estimate the proportion of HIV positive individuals in the population using the estimator in \eqref{eq:mmest}.
	\end{itemize}

	In each network, we simulate the transmission and sampling process and calculate the prevalence estimator described in Theorem \ref{thm:moments}, repeating this process 1000 times. A description of these results follows in figure ~\ref{fig3} on page ~\pageref{fig3} and table~\ref{tab2} in page ~\pageref{tab2}.
	
	\begin{figure}
		\begin{centering}
			{\includegraphics[width=0.47\linewidth]{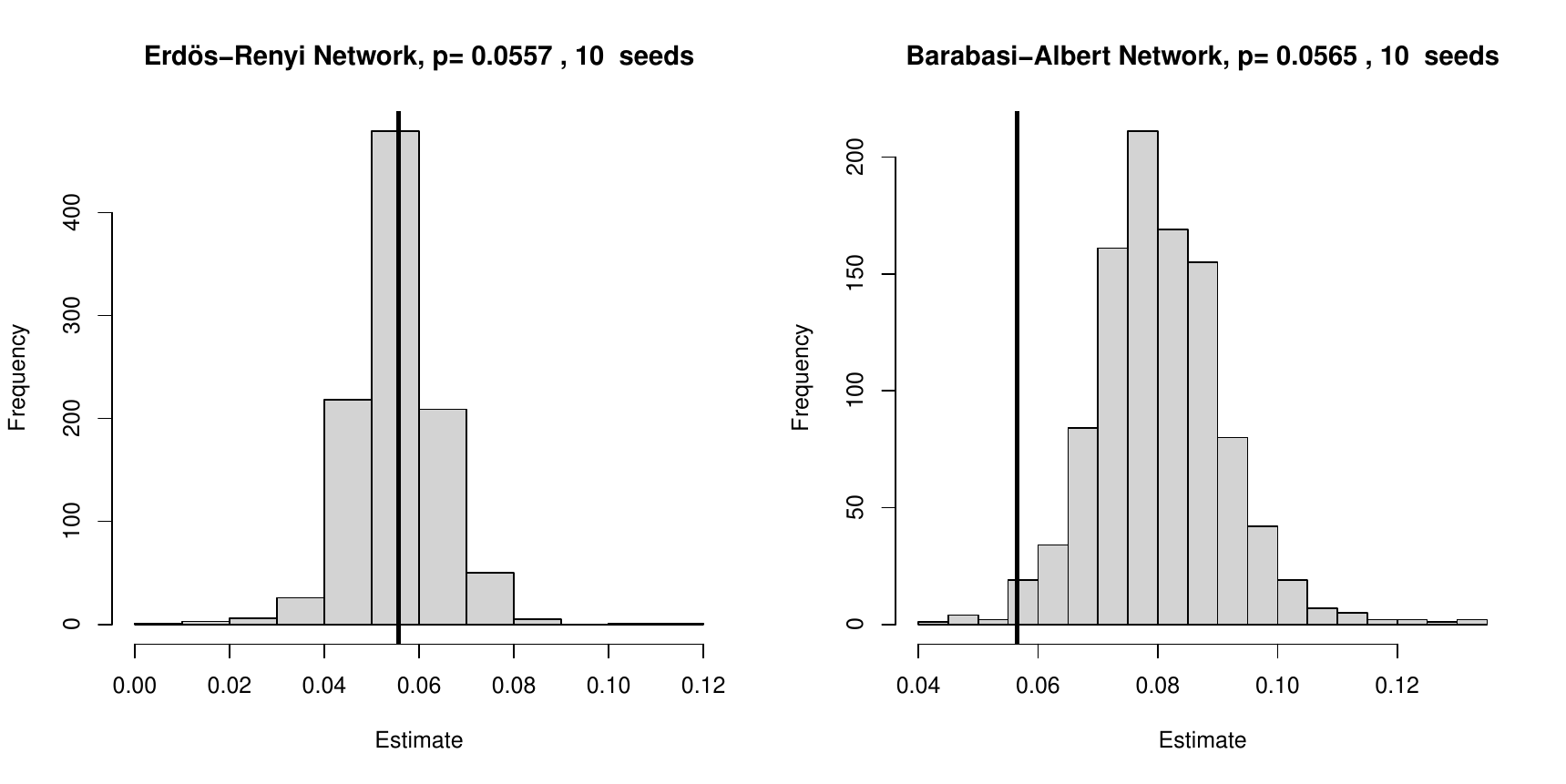}}
			{\includegraphics[width=0.47\linewidth]{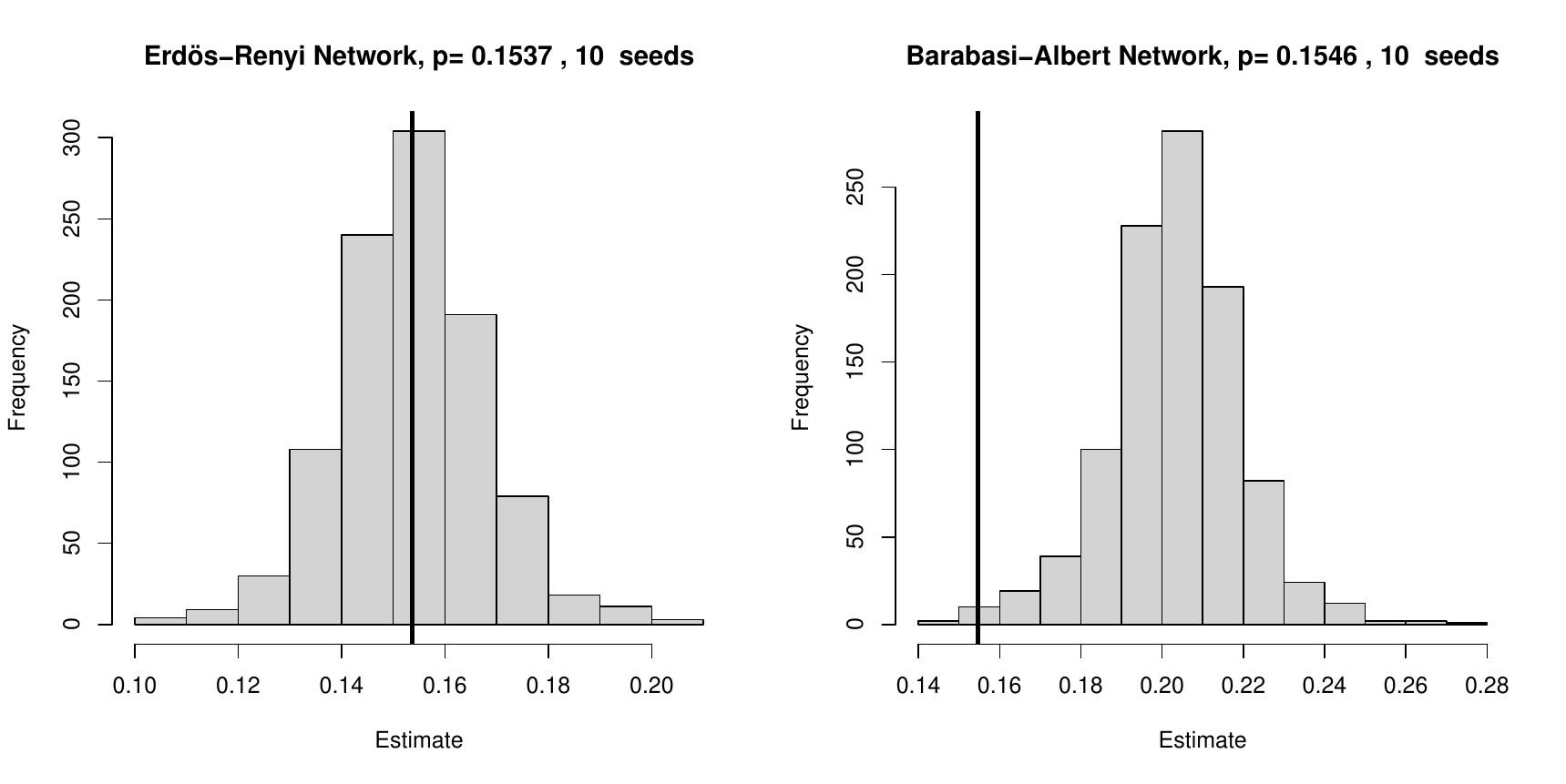}}\\
			{\includegraphics[width=0.47\linewidth]{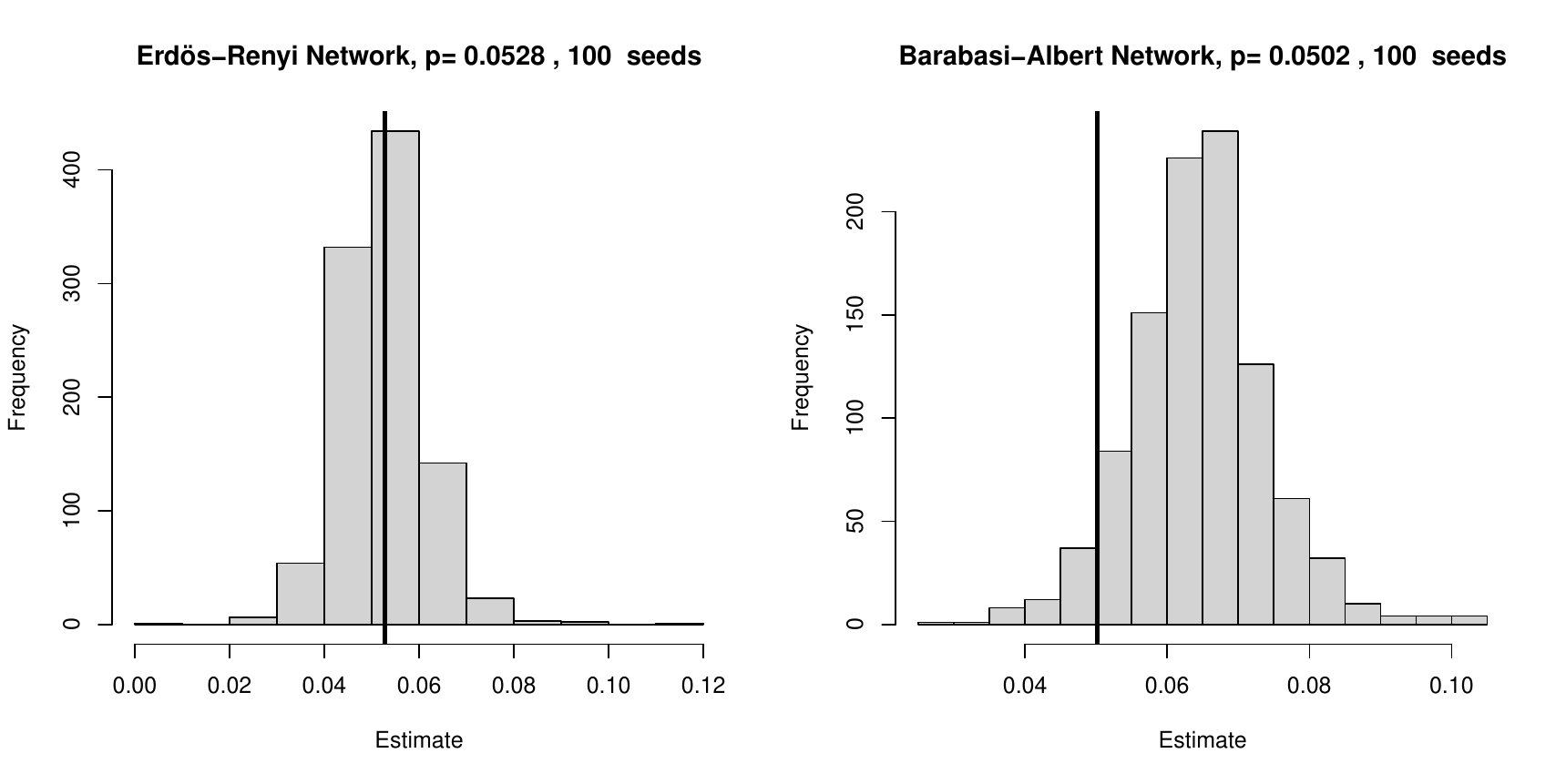}}
			{\includegraphics[width=0.47\linewidth]{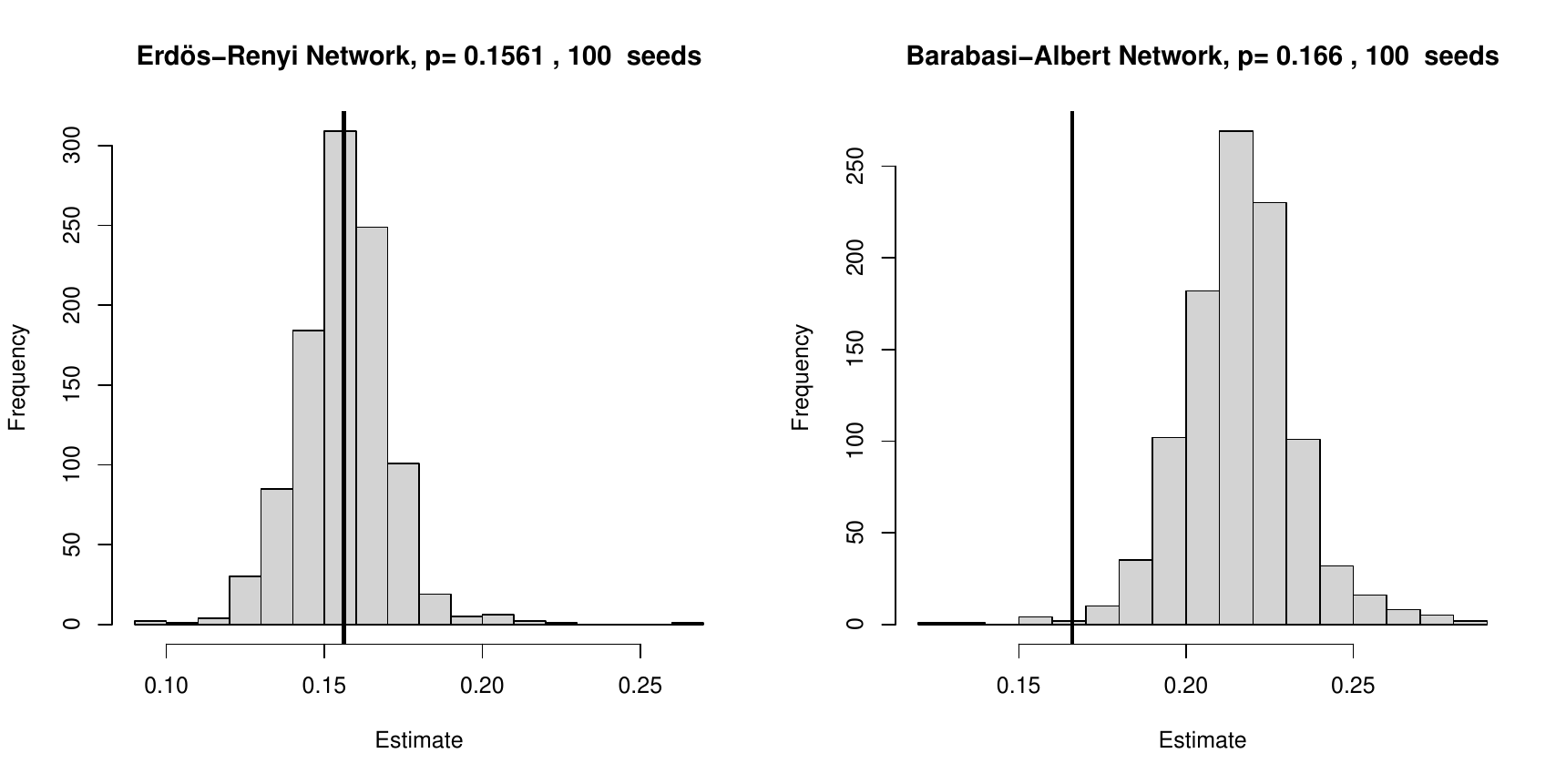}}\\
			{\includegraphics[width=0.47\linewidth]{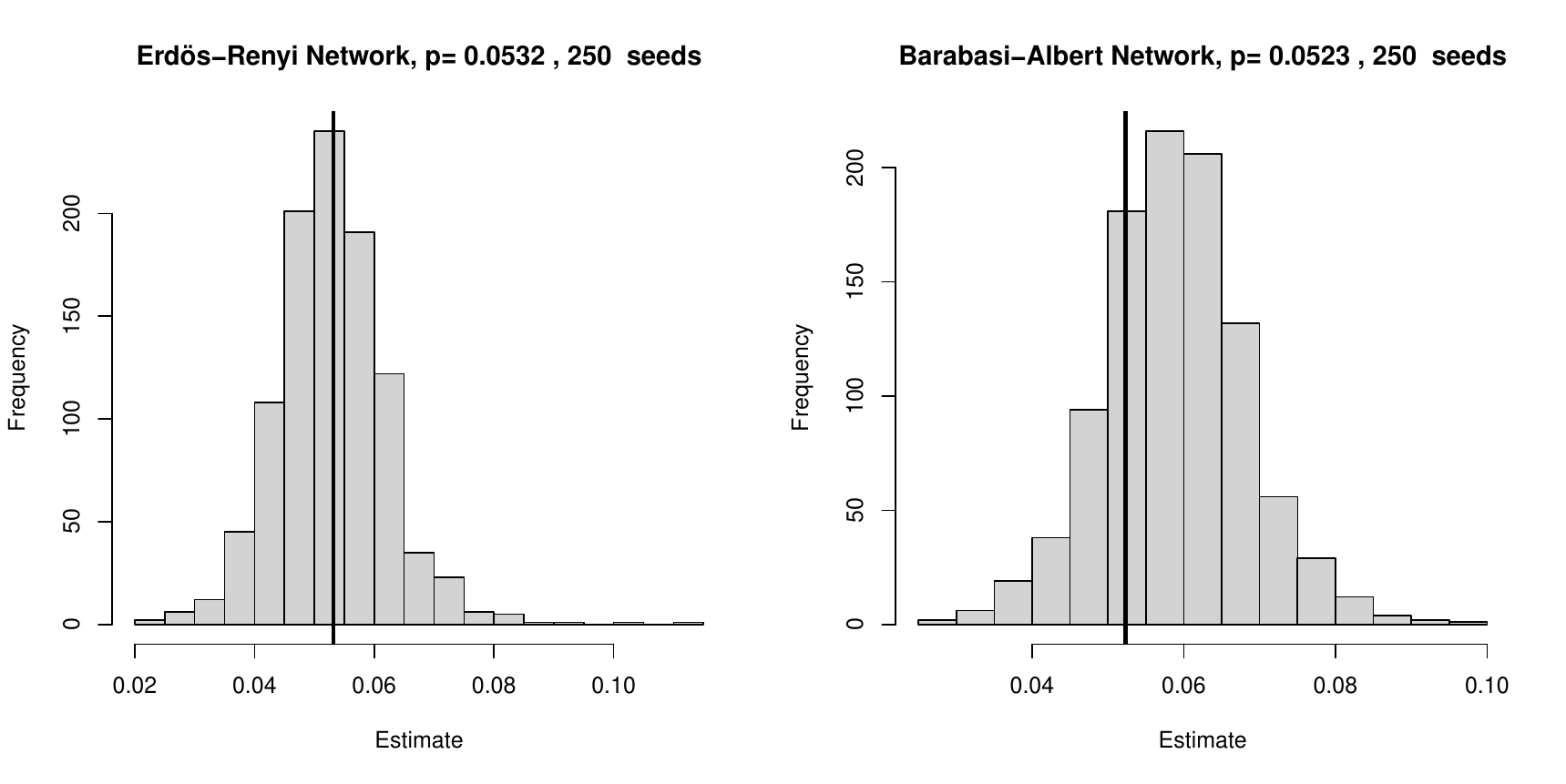}}
			{\includegraphics[width=0.47\linewidth]{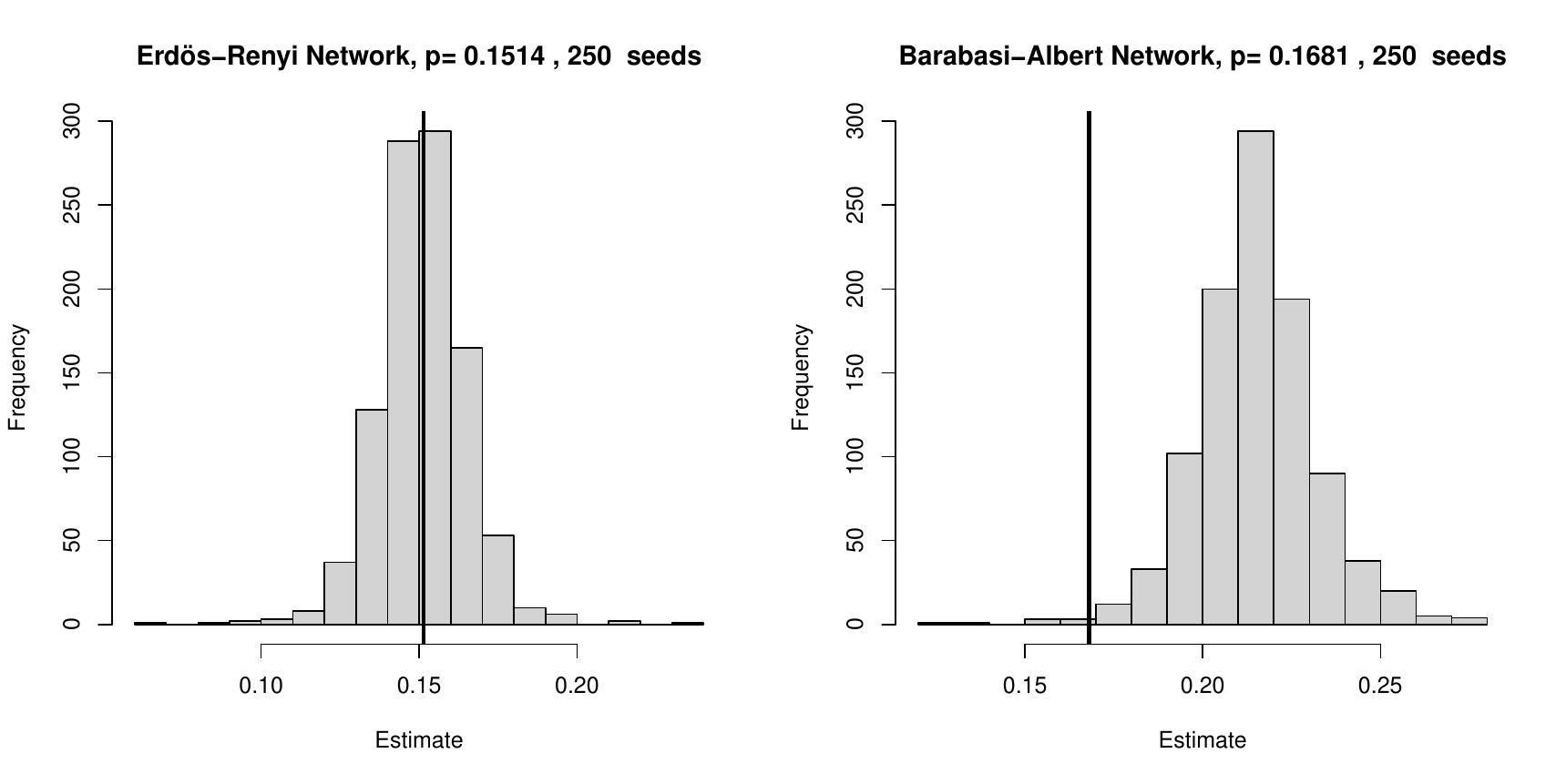}}\\
		\end{centering}
		\caption{Each pair of histograms contain the empirical distribution of estimates for the proportion of HIV status based on 1000 simulations. Number of seeds used are 10 (top row), 100 (middle row) and 250 (bottom row); transmission is spread until $5\%$ (left column) and $15\%$ (right column) of population reaches positive HIV status.}
		\label{fig3}
	\end{figure}

\begin{table}[h]
		\caption{Theoretical and average values (with standard deviations) for largest eigenvalue and normalized corresponding eigenvector.}
	\label{tab2}
	\centering
	\begin{tabular}{lllllll}
		\hline
		\multicolumn{1}{|l|}{Network}            & \multicolumn{6}{l|}{Erdös-Renyi}                                                                                                                                  \\ \hline
		\multicolumn{1}{|l|}{Target prevalence}  & \multicolumn{3}{l|}{$p\sim 5 \%$}                                             & \multicolumn{3}{l|}{$p\sim 15 \%$}                                                \\ \hline
		\multicolumn{1}{|l|}{Number of seeds}    & \multicolumn{1}{l|}{10} & \multicolumn{1}{l|}{100} & \multicolumn{1}{l|}{250} & \multicolumn{1}{l|}{10} & \multicolumn{1}{l|}{100} & \multicolumn{1}{l|}{250}     \\ \hline
		\multicolumn{1}{|l|}{Actual prevalence}  & 0.0557                  & 0.0528                   & 0.0532                   & 0.1537                  & 0.1561                   & \multicolumn{1}{l|}{0.1514}  \\ \cline{1-1}
		\multicolumn{1}{|l|}{Estimator value}    & 0.0552                  & 0.0526                   & 0.053                    & 0.1534                  & 0.1561                   & \multicolumn{1}{l|}{0.1511}  \\ \cline{1-1}
		\multicolumn{1}{|l|}{Bias}               & -0.0005                 & -0.0002                  & -0.0002                  & -0.0003                 & 0                        & \multicolumn{1}{l|}{-0.0003} \\ \cline{1-1}
		\multicolumn{1}{|l|}{Standard deviation} & 0.0095                  & 0.009                    & 0.00917                  & 0.0144                  & 0.0148                   & \multicolumn{1}{l|}{0.01424} \\ \hline
		&                         &                          &                          &                         &                          &                              \\ \hline
		\multicolumn{1}{|l|}{Network}            & \multicolumn{6}{l|}{Barabasi-Albert}                                                                                                                              \\ \hline
		\multicolumn{1}{|l|}{Target prevalence}  & \multicolumn{3}{l|}{$p\sim 5 \%$}                                             & \multicolumn{3}{l|}{$p\sim 15 \%$}                                                \\ \hline
		\multicolumn{1}{|l|}{Number of seeds}    & \multicolumn{1}{l|}{10} & \multicolumn{1}{l|}{100} & \multicolumn{1}{l|}{250} & \multicolumn{1}{l|}{10} & \multicolumn{1}{l|}{100} & \multicolumn{1}{l|}{250}     \\ \hline
		\multicolumn{1}{|l|}{Actual prevalence}  & 0.0565                  & 0.0502                   & 0.0523                   & 0.1546                  & 0.166                    & \multicolumn{1}{l|}{0.1681}  \\ \cline{1-1}
		\multicolumn{1}{|l|}{Estimator value}    & 0.0804                  & 0.0645                   & 0.0587                   & 0.203                   & 0.216                    & \multicolumn{1}{l|}{0.2152}  \\ \cline{1-1}
		\multicolumn{1}{|l|}{Bias}               & 0.0239                  & 0.0143                   & 0.0064                   & 0.04843                 & 0.05                     & \multicolumn{1}{l|}{0.0471}  \\ \cline{1-1}
		\multicolumn{1}{|l|}{Standard deviation} & 0.0109                  & 0.0096                   & 0.00949                  & 0.0162                  & 0.0192                   & \multicolumn{1}{l|}{0.0169}  \\ \hline
	\end{tabular}
\end{table}
	
	The proportion estimates have a normal shaped empirical distribution in all simulated scenarios, in accordance to Theorem \ref{thm:moments}, and they have a negligible bias when obtained from an Erdös-Renyi network. When data comes from a Barabasi-Albert network, however, the estimators in \eqref{eq:mmest} consistently overestimate the true population prevalence, particularly when the seed size is small, without any noticeable change in standard error. A large seed size compared to the target population prevalence resembles transmission at random, and in this case our estimates have a smaller (while still noticeable) bias as can be seen in the lower left panel of figure \ref{fig3} for $K=250$ and $p=5\%$. Conversely, for smaller seed size, transmission strongly depends on the network structure and the method of moments overestimates the true proportion in over $95\%$ of cases (see top row of figure \ref{fig3}, where transmission is initiated with only 10 individuals), which may indicate that the assumption that all HIV positive individuals have the same offspring (survey) distribution is not valid. Prior studies have incorporated network information into the performance of RDS based estimators \cite{mccreesh2012evaluation,wejnert2009empirical,mills2014errors}. Considering such information in our proposed estimates, however, would require extending the types of individuals to pose an offspring probability distribution according not only to HIV status but also to network size and structure. As \cite{sperandei2018assessing} points out, the number of contacts in common between  individuals is hard to assess, and there is little, if any, information about it, so even assuming that network information from different individuals is precise, their overlap is usually hard or impossible to estimate in real-life conditions.
	
	
	\section{Conclusions}\label{sec:disc}
	
	This study represents a significant extension of the work conducted by \cite{mm,maki1989can} on sampling distributions and estimation for family size data to the multi-dimensional branching process framework. Our research delves into various dimensions, yielding outcomes that parallel those of the single-dimensional case. Moreover, we have gained insights into numerous properties and applications that align with the findings of the singular case, contributing to a deeper understanding of sampling distributions and estimation techniques for multi-dimensional branching processes, thereby enriching the existing body of knowledge in this field.

	The limiting result provided by \cite{ksarticle} for supercritical branching processes that almost surely do not go extinct and reasonable assumptions on the offspring distributions involved in the branching process provides a sample size for which the probability of finding no siblings in a random sample grows to one, a condition under which the sampled family sizes is shown to behave as independent and identically distributed data with an asymptotic probability mass function depending on its mean offspring matrix and the probability mass functions of each of the types of individuals in the branching process. Parametric estimation methods based on maximum likelihood and mean offspring matrix characteristics moment estimation methods based on this asymptotic distribution are shown to be consistent.
	
	Assumptions on the offspring behavior are less restricting than in the one dimensional case, in the sense that supercriticality is demanded only from some type of individuals to guarantee the global supercriticality of the Markov Chain. The largest eigenvalue being greater than one is a condition that can be achieved without requiring that all types are supercritical. This allows for study of processes that are not necessarily supercritical by themselves but are a coordinate of a supercritical multitype branching process, as is the case with the first three coordinates of the mentioned example. The allowed sampling size, mentioned in assumption \ref{ass:cond} is bounded by the powers of this largest eigenvalue, again allowing for a larger sample of a multidimensional process than the marginal study of single components of interest would.
	
	The numerical study conducted in this work shows that estimation of population features based on $p_S$ is consistent and asymptotically normal under ideal circumstances (individuals connected at random). The preferential attachment model of the Barabasi-Albert network, however, is a more realistic social network model \cite{jusup2022social,simon1955class} and our estimates show a systematic bias in this case. Utilizing self reported network distribution data has well documented problems \cite{mccreesh2012evaluation, wejnert2009empirical}, and HIV prevalence literature suggests that network distribution may rely not only on HIV status but also on social and demographic features (neighborhood, national origin, sexual orientation) \cite{montealegre2012hiv,montealegre2012prevalence}. The accuracy of our method of moments estimates in the connected at random setting suggests that accomodating the branching process to include offspring distributions for each feature will produce reliable estimates based on Respondent Driven Sampling using real data based on social networks.
	
	Although this work allows for the characterization of $p_S$, estimates for the distribution of each of the types' offspring distribution are not available under the absence of parametric assumptions on them. This is a natural consequence of sampling a single generation, the setting of this work, so future developments include the extension of the framework in this paper to the sampling of multiple generations and/or to sampling schemes that include information on the parent type of a sampled vector of family sizes, although the latter may result in mathematically trivial developments that require hard to obtain information, impractical in an applied setting.
	
	Future research should include the extension or modification of this work to allow for offspring distributions that encompass the zero vector within their support. Banning the case of zero offspring is a sufficient condition for non-extinction, but while our analysis has demonstrated that the exclusion of zero offspring is also adequate to ensure the convergence of the sampling distribution to an i.i.d. mixture, it is pertinent to also study the sampling distribution of real-world branching processes that do not inherently adhere to this condition \cite{davison2021parameter, couronne2014branching}. Elements within a branching process that yield zero offspring would remain unobservable, akin to censored data, requiring the use of methods like expectation-maximization would have to be applied in order to make inference on such distributions. By addressing these complexities, future studies aim to refine our understanding of branching processes and enhance their applicability in diverse real-world contexts.

	\bibliographystyle{tfnlm}
	\bibliography{paper}

	\appendix
	\section{Proof of lemmas and theorems}\label{sec:proof}
	
	Throughout this section, we will appeal to Corollary \ref{cor:rates}, which is an immediate consequence of the  Kesten-Stigum Theorem (Theorem \ref{thm:KST})
	
	\begin{cor}\label{cor:rates} The following asymptotic results hold almost surely, with $W$ the one dimensional random variable defined in Theorem \ref{thm:KST}
		
		\begin{itemize}
			\item[i)] $$|\bm{Z_{n}}|\rho^{-n} \rightarrow W .$$
			\item[ii)] $$Z_{n,i}\rho^{-n} \rightarrow b_iW \mbox{, for } i\in \{1,\ldots , l \}.$$
			\item[iii)] $$\frac{Z_{n,i}}{|\bm{Z_{n}}|}  \rightarrow b_i \mbox{, for } i\in \{1,\ldots , l \}.$$
			\item[iv)] For  $k \in \mathbb{N}$ $$\frac{|\bm{Z_{n-1}}|^k}{|\bm{Z_{n}}|^k}\rightarrow \rho^{-k} \mbox{ and for } i\in \{1,\ldots , l \},  \frac{Z_{n-1,i}^k}{Z_{n,i}^k} \rightarrow \rho^{-k}.$$
		\end{itemize}
		
	\end{cor}
	
	We also establish two calculus results regarding the rate of growth of $r_n$:
	
	\begin{lem}\label{lem:KSc} Under assumption \ref{ass:cond}, $r^2_n|\bm{Z_{n-1}}|^{-1}\rightarrow 0$ almost surely as $n \to \infty$
	\end{lem}
	\begin{proof}
		$$\frac{r^2_n}{|\bm{Z_{n-1}}|}=\frac{r^2_n}{\rho^n}\frac{\rho^{n-1}}{|\bm{Z_{n-1}}|}\rho $$
		
		By Corollary \ref{cor:rates} of Theorem \ref{thm:KST}, $\frac{\rho^{n-1}}{|\bm{Z_{n-1}}|}$ converges almost surely to $W^{-1}$. As the sequence is a martingale, Doob's martingale convergence  theorem \cite{doob1961notes} and Jensen's inequality imply that $W^{-1}$ is bounded almost surely, and since by assumption \ref{ass:cond} i) $\frac{r^2_n}{\rho^n}\rightarrow 0$, the result follows.
	\end{proof}
	
	\begin{lem}\label{lem:aux}
		Let $K$ as in assumption \ref{ass:cond} iii). For $\alpha \in (0, -\log_\rho K)$, $\rho^{\alpha n}\E(|\bm{Z_{n}}|^{-1})\rightarrow 0.$
		
	\end{lem}

	\begin{proof} As the arithmetic mean of positive numbers is always greater than their geometric mean
		
		\begin{align*}
			\E(|\bm{Z_{n}}|^{-1}|\bm{Z_{n-1}})&=\E\Biggl[\Bigg(\sum_{i=1}^l\sum_{j=1}^{Z_{n-1,i}}|Y^{(i)}_j|\Bigg)^{-1}\Bigg|\bm{Z_{n-1}}\Biggr] \\
			&\leq|\bm{Z_{n-1}}|^{-1}  \E\Biggl[\Bigg(\prod_{i=1}^l\prod_{j=1}^{Z_{n-1,i}}|Y^{(i)}_j|\Bigg)^{-1/|\bm{Z_{n-1}}|}\Bigg|\bm{Z_{n-1}}\Biggr].
		\end{align*}
		
		From Jensen's conditional inequality, independence and assumption \ref{ass:cond} iii), one gets
		
		\begin{align*}
			\E(|\bm{Z_{n}}|^{-1}|\bm{Z_{n-1}})&\leq|\bm{Z_{n-1}}|^{-1}  \E\Biggl[\prod_{i=1}^l\prod_{j=1}^{Z_{n-1,i}}|Y^{(i)}_j|\Bigg|\bm{Z_{n-1}}\Biggr]^{-1/|\bm{Z_{n-1}}|}\\
			&\leq |\bm{Z_{n-1}}|^{-1}\prod_{i=1}^l\prod_{j=1}^{Z_{n-1,i}} \E(|Y^{(i)}_j|^{-1})^\frac{1}{|\bm{Z_{n-1}}|}\\
			&\leq |\bm{Z_{n-1}}|^{-1}\prod_{i=1}^l\prod_{j=1}^{Z_{n-1,i}} K^\frac{1}{|\bm{Z_{n-1}}|}\\
			&=|\bm{Z_{n-1}}|^{-1} K.\\
		\end{align*}
		
		Taking expected values iteratively, one gets the result as

		$$	0\leq 	\rho^{\alpha n}\E(|\bm{Z_{n}}|^{-1}) \leq \rho^{\alpha (n-1)}\E(|\bm{Z_{n-1}}|^{-1})\rho^\alpha K \leq\cdots \leq \E(|Z_0|^{-1})[\rho^\alpha K]^n \to 0,$$
		and the last term goes to zero since $\alpha <-\log_\rho K \implies \rho^\alpha K < 1$.
		
	\end{proof}
	
	\subsection{Proof of Lemma \ref{lem:sam2}} \begin{proof}
		
		We consider the following event throughout this analysis of the sampling distribution:
		
		$$A=\{\mbox{ sampled individuals $1$ and $2$ have parents of the same type}\}$$
		
		We split this analysis into three scenarios:

\indent $\bullet$\textbf{Case 1}:$\bm{u}=\bm{v}=\bm{e_k}\in \mathbb{R}^l$.
			This is the case where two families consisting of one individual of type $k$ are selected, for some $k \in \{1,\ldots , l\}$. We have

			$$	\P(X_1=\bm{u}, X_2=\bm{v}|\bm{Z_{n-1}})=\P(X_1= X_2=\bm{e_k}, A|\bm{Z_{n-1}})+\P(X_1= X_2=\bm{e_k}, A^c|\bm{Z_{n-1}})$$
			We now analyze this two terms separately. On $A$, using Bayes' formula:
			
			\begin{align*}
				\P(X_1= X_2=\bm{e_k}, A|\bm{Z_{n-1}})=& \sum_{i=1}^l \P(\mbox{ two $\bm{e_k}$ families are selected with parents of type $i$})\\
				=&\sum_{i=1}^l \P(\mbox{ two $\bm{e_k}$ families are selected}|\mbox{parents of type $i$})\times \cdots \\ &\cdots \P(\mbox{selected families w/parents of type $i$})\\
				=&\sum_{i=1}^l \P(\mbox{ two $\bm{e_k}$ families are selected}|\mbox{parents of type $i$})\times \cdots \\ &\cdots \E\Biggl(\frac{S_{n-1,i}(S_{n-1,i}-1)}{|\bm{Z_{n}}|(|\bm{Z_{n}}|-1)}\biggr|\bm{Z_{n-1}}\Biggr).\\
			\end{align*}

			To study $(*)=\P(\mbox{ two $\bm{e_k}$ families are selected}|\mbox{parents of type $i$})$, define the events: $B_{s,t}=\{\mbox{families $s$ and $t$ are $\bm{e_k}$}\}, C_{s,t}=\{\mbox{families $s$ and $t$ are selected}\}.$ With these,
			
			\begin{align*}
				(*)&=\sum_{t=1}^{Z_{n-1,i}} \sum_{s>t}\P(C_{s,t},B_{s,t}|\mbox{parents of type $i$})\\
				&=\sum_{t=1}^{Z_{n-1,i}} \sum_{s>t}\P(C_{s,t}|B_{s,t},\mbox{parents of type $i$})\P(B_{s,t}|\mbox{parents of type $i$})\\
				&=\sum_{t=1}^{Z_{n-1,i}} \sum_{s>t}\E\Biggl(\binom{2+\sum_{j=3}^{Z_{n-1,i}}|Y_{n-1,j}^{(i)}|}{2}^{-1}\bigg|\bm{Z_{n-1}}\Biggr)p_i(\bm{e_k})^2\\
				&=\binom{Z_{n-1,i}}{2}\E\Biggl(\binom{2+\sum_{j=3}^{Z_{n-1,i}}|Y_{n-1,j}^{(i)}|}{2}^{-1}\bigg|\bm{Z_{n-1}}\Biggr)p_i(\bm{e_k})^2.\\
			\end{align*}
			So the term $\P(X_1= X_2=\bm{e_k}, A|\bm{Z_{n-1}})$ is equal to $$\sum_{i=1}^l\E\Biggl( \frac{S_{n-1,i}(S_{n-1,i}-1)}{|\bm{Z_{n}}|(|\bm{Z_{n}}|-1)}\binom{Z_{n-1,i}}{2}\binom{2+\sum_{j=3}^{Z_{n-1,i}}|Y_{n-1,j}^{(i)}|}{2}^{-1}\bigg|\bm{Z_{n-1}}\Biggr)p_i(\bm{e_k})^2.$$
			
			Analogously for $\P(X_1= X_2=\bm{e_k}, A^c|\bm{Z_{n-1}})$, we have,
			
			\begin{align*}
				\P(X_1= X_2=\bm{e_k}, A^c|\bm{Z_{n-1}})&=2\sum_{i=1}^l\sum_{j>i}\P(\mbox{two $\bm{e_k}$ families selected, parents of type $i$,$j$})\\
				&=2\sum_{i=1}^l\sum_{j>i}\P(\mbox{$\bm{e_k}$ family selected}|\mbox{parent type $i$})\times \cdots \\ &\cdots \frac{S_{n-1,i}}{|\bm{Z_{n}}|}\P(\mbox{$\bm{e_k}$ family selected}|\mbox{parent type $j$})\frac{S_{n-1,j}}{|\bm{Z_{n}}|-1}.\\
			\end{align*}
			
			For arbitrary type $i$ one has
			
	$$\P(\mbox{$\bm{e_k}$ family selected}|\mbox{parent type $i$})=\sum_{t=1}^{Z_{n-1,i}}\P(\mbox{family $t$ is selected and is $\bm{e_k}$}|\mbox{parent type $i$})$$
	
	Where the probability of a specific family having a single individual of type $k$ and being sampled given the type of parents can be further expressed as a product of conditional probabilities $\P(\mbox{family $t$ selected}|\mbox{family $t$ is $\bm{e_k}$, parent type $i$})\times \P(\mbox{family $t$ is $\bm{e_k}$}|\mbox{parent type $i$})$, so the term $\P(X_1=\bm{u}, X_2=\bm{v}, A|\bm{Z_{n-1}})$ is equal to 
			
			$$ 2\sum_{i=1}^l\sum_{j>i} \frac{Z_{n-1,i}Z_{n-1,j}}{|\bm{Z_{n}}|(|\bm{Z_{n}}|-1)} p_i(\bm{e_k})p_j(\bm{e_k})\E\Biggl(\frac{S_{n-1,i}}{1+\sum_{w=2}^{Z_{n-1,i}}|Y_{n-1,w}^{(i)}|}\frac{S_{n-1,j}}{1+\sum_{w=2}^{Z_{n-1,j}}|Y_{n-1,w}^{(j)}|}\bigg|\bm{Z_{n-1}} \Biggr),$$ which leads to
			
			\begin{multline}\label{eq:a1}
				\P(X_1=\bm{e_k},X_2=\bm{e_k}|\bm{Z_{n-1}})= \sum_{i=1}^l\frac{S_{n-1,i}}{|\bm{Z_{n}}|}p_i(\bm{e_k}) \times \\ \Biggl[ \frac{S_{n-1,i}-1}{|\bm{Z_{n}}|-1}\binom{Z_{n-1,i}}{2}\E\biggl( \binom{2+\sum_{j=3}^{Z_{n-1,i}}|Y_{n-1,j}^{(i)}|}{2}^{-1}\bigg|\bm{Z_{n-1}}\biggr)p_i(\bm{e_k}) + \\ \sum_{j>i}\frac{2}{|\bm{Z_{n}}|-1}\E\Biggl(\frac{Z_{n-1,i}Z_{n-1,j}}{1+\sum_{w=2}^{Z_{n-1,i}}|Y_{n-1,w}^{(i)}|}\frac{S_{n-1,j}}{1+\sum_{w=2}^{Z_{n-1,j}}|Y_{n-1,w}^{(j)}|}\bigg|\bm{Z_{n-1}} \Biggr)p_j(\bm{e_k}) \Biggr].
			\end{multline}

\indent$\bullet$\textbf{Case 2}:$\bm{u}\neq \bm{v}$.
			
			As $\bm{u}\neq \bm{v}$ does not allow sampled families to be the same, a very similar analysis as the previous case can be carried out, splitting the probability into same type parents and different type parents to get
			
			\begin{multline}\label{eq:a2}
				\P(X_1=\bm{u},X_2=\bm{v}|\bm{Z_{n-1}})= \frac{|\bm{u}||\bm{v}|}{|\bm{Z_{n}}|(|\bm{Z_{n}}|-1)}\sum_{i=1}^lp_i(\bm{u}) \Biggl[ \\ \binom{Z_{n-1,i}}{2}\E\biggl( S_{n-1,i}(S_{n-1,i}-1)\binom{|\bm{u}|+|\bm{v}|+\sum_{j=3}^{Z_{n-1,i}}|Y_{n-1,j}^{(i)}|}{2}^{-1}\bigg|\bm{Z_{n-1}}\biggr)p_i(\bm{v})\\ + \sum_{j>i}2Z_{n-1,i}Z_{n-1,j}\E\biggl(\frac{S_{n-1,i}}{|\bm{u}|+\sum_{w=2}^{Z_{n-1,i}}|Y_{n-1,w}^{(i)}|}\frac{S_{n-1,j}}{|\bm{v}|+\sum_{w=2}^{Z_{n-1,j}}|Y_{n-1,w}^{(j)}|} \bigg|\bm{Z_{n-1}}\biggr)p_j(\bm{v}) \Biggr].
			\end{multline}

\indent$\bullet$\textbf{Case 3}$\bm{u}=\bm{v} \neq \bm{e_k}\in \mathbb{R}^l$.
			
			We must take into account the possibility of both sampled individuals belonging to the same family. As this only happens in $A$, the analysis in $A^c$ is identical to the previous cases and yields that the sampling probability $
			\P(X_1=X_2=\bm{u}, A^c|\bm{Z_{n-1}})	= 2\sum_{i=1}^l\sum_{j>i}p_i(\bm{u})p_j(\bm{u}) \E\Biggl(\frac{Z_{n-1,i}S_{n-1,i}Z_{n-1,j}S_{n-1,j}}{|\bm{Z_{n}}|(|\bm{Z_{n}}|-1)}\Biggl[\frac{|\bm{u}|}{|\bm{u}|+\sum_{w=2}^{Z_{n-1,i}}|Y_{n-1,w}^{(i)}|}\Biggr]^2  \bigg|\bm{Z_{n-1}}\Biggr)$.
			
			On $A$, we define the set $F=\{ \mbox{individuals $1$ and $2$ come from different families.}\}$. On $F$, the analysis can be performed as in the previous two cases and leads to $\P(X_1=X_2= \bm{u}, A,F|\bm{Z_{n-1}}) =\sum_{i=1}^lp_i(\bm{u})^2|\bm{u}|^2 \E\Biggl( \frac{S_{n-1,i}(S_{n-1,i}-1)}{|\bm{Z_{n}}|(|\bm{Z_{n}}|-1)}\binom{Z_{n-1,i}}{2}\binom{2|\bm{u}|+\sum_{j=3}^{Z_{n-1,i}}|Y_{n-1,j}^{(i)}|}{2}^{-1} \bigg|\bm{Z_{n-1}}\Biggr)$.
			
			We use total probability on $\P(X_1=\bm{u}=X_2, A,F^c|\bm{Z_{n-1}})$ conditioning on the following events: $B_t=\{\mbox{family $t$ is selected twice} \},C_t=\{\mbox{family $t$ is } \bm{u} \}.$ Then, $\P(X_1=X_2= \bm{u}, A,F^c|\bm{Z_{n-1}})$ equals
			
			\begin{align*}
				&\sum_{i=1}^l \P(X_1=X_2=\bm{u}, A,F^c|\mbox{parent type i},\bm{Z_{n-1}})\P(\mbox{parent type i}|\bm{Z_{n-1}})\nonumber\\
				&=\sum_{i=1}^l\sum_{t=1}^{Z_{n-1,i}} \P(B_t,C_t|\mbox{parent type i},\bm{Z_{n-1}})\frac{S_{n-1,i}}{|\bm{Z_{n}}|}\nonumber\\
				&=\sum_{i=1}^l\sum_{t=1}^{Z_{n-1,i}} \P(B_t|C_t,\mbox{parent type i},\bm{Z_{n-1}})\P(C_t|\mbox{parent type i},\bm{Z_{n-1}})\frac{S_{n-1,i}}{|\bm{Z_{n}}|}\nonumber\\
				&=\sum_{i=1}^l\sum_{t=1}^{Z_{n-1,i}}\E\Biggl(\binom{|\bm{u}|}{2}\binom{|\bm{u}|+\sum_{w=2}^{Z_{n-1,i}}|Y_{n-1,w}^{(i)}|}{2}^{-1}\bigg|\bm{Z_{n-1}} \Biggr)p_i(\bm{u})\frac{S_{n-1,i}}{|\bm{Z_{n}}|}\nonumber\\
				&=\sum_{i=1}^l \E\Biggl(\binom{|\bm{u}|}{2}\binom{|\bm{u}|+\sum_{w=2}^{Z_{n-1,i}}|Y_{n-1,w}^{(i)}|}{2}^{-1}\bigg|\bm{Z_{n-1}} \Biggr)p_i(\bm{u})\frac{Z_{n-1,i}S_{n-1,i}}{|\bm{Z_{n}}|}.
			\end{align*}
		
		We therefore have
		\begin{multline}\label{eq:a3}
			P(X_1=X_2=\bm{u}|\bm{Z_{n-1}})=\sum_{i=1}^lp_i(\bm{u})|\bm{u}|S_{n-1,i}|\bm{Z_{n-1}}|^{-1}\Biggl[\\
			p_i(\bm{u})|\bm{u}|\frac{S_{n-1,i}-1}{|\bm{Z_{n}}|-1}\binom{Z_{n-1,i}}{2} \E\Biggl( \binom{2|\bm{u}|+\sum_{j=3}^{Z_{n-1,i}}|Y_{n-1,j}^{(i)}|}{2}^{-1} \bigg|\bm{Z_{n-1}}\Biggr)
			+\\ \sum_{j>i}p_j(\bm{u})\frac{Z_{n-1,i}Z_{n-1,j}S_{n-1,j}}{(|\bm{Z_{n}}|-1)} \E\Biggl(\frac{|\bm{u}|}{(|\bm{u}|+\sum_{w=2}^{Z_{n-1,i}}|Y_{n-1,w}^{(i)}|)^2}  \bigg|\bm{Z_{n-1}}\Biggr)+\\
			Z_{n-1,i}\frac{|\bm{u}|-1}{2}\E\Biggl(\binom{|\bm{u}|+\sum_{w=2}^{Z_{n-1,i}}|Y_{n-1,w}^{(i)}|}{2}^{-1}\bigg|\bm{Z_{n-1}} \Biggr)\Biggr]\\
		\end{multline}
		
		The three cases presented in \eqref{eq:a1}, \eqref{eq:a2} and \eqref{eq:a3} encompass all possibilities. \end{proof}
	
	\subsection{Proof of Theorem \ref{thm:assympDn}}
	
	We recall the Vandermonde convolution identity \cite{vdm}:
	
	\begin{lem}\label{lem:vdm}
		Vandermonde identity: For integers $a_1, \ldots, a_p, n$,
		
		$$\binom{\sum_{k=1}^p a_k}{n}=\sum_{|b|=n}\prod_{k=1}^p \binom{a_k}{b_k}.$$
		
	\end{lem}
	
	\begin{proof}
		We begin by observing that the probability of selecting $r_n$ different families when all $p_i$'s concentrate at the same vector $\bm{v}$ is a decreasing function of $|\bm{v}|$. Indeed, denote this probability by $P_v(D_n)$. As the function $f(x)=\frac{x^n-ax}{x^n-a}$ is decreasing for $x\geq \Biggl(\frac{r_n-1}{n}\Biggr)^{\frac{1}{n-1}}$, we have that $\Biggl(\frac{r_n-1}{n}\Biggr)^{\frac{1}{n-1}}\leq|\bm{u}|\leq |\bm{v}|$ it holds that
		
		\begin{align*}
			P_u(D_n)&=\E\Biggl(\biggr[\sum_{s \in \mathbb{N}^l: |s|=r_n}\prod_{i=1}^{l} \binom{Z_{n-1,i}}{s_i}|\bm{u}|^{s_i}\biggr]\binom{|\bm{Z_{n}}|}{r_n}^{-1}\Biggr)\\
			&=|\bm{u}|^{r_n}\E\Biggl(\biggr[\sum_{s \in \mathbb{N}^l: |s|=r_n}\prod_{i=1}^{l} \binom{u_i|\bm{u}|^{n-2}}{s_i}\biggr]\binom{|\bm{u}|^n}{r_n}^{-1}\Biggr)\\
			&=|\bm{u}|^{r_n}\binom{|\bm{u}|^{n-1}}{r_n}\binom{|\bm{u}|^n}{r_n}^{-1}\\
			&=\prod_{a=0}^{r_n-1}\frac{|\bm{u}|^n-a|\bm{u}|}{|\bm{u}|^n-a}\\
			&\geq \prod_{a=0}^{r_n-1}\frac{|\bm{v}|^n-a|\bm{v}|}{|\bm{v}|^n-a}= \cdots = P_v(D_n).
		\end{align*}
	We consider initially the case of offspring distributions of bounded support in $\mathbb{N}^l$ for all $l$ types. Define $\Omega_i = \{\bm{w}: p_i(\bm{w})>0 \}$ the support of distribution $p_i$; $\Omega = \cup_{i=1}^l \Omega_i$  and denote by $ \bm{u}\in\arg\max_{\bm{w} \in \Omega} |\bm{w}|; \bm{v}\in \arg\min_{\bm{w} \in \Omega} |\bm{w}|$ the composition of the smallest and largest possible offspring of any type (these vectors need not be unique. However, only their sizes are relevant in the following argument). The argument made at the very beginning guarantees that 
		
		$$P_u(D_n|\bm{Z_{n-1}})\leq \P(D_n|\bm{Z_{n-1}})\leq P_v(D_n|\bm{Z_{n-1}})\leq 1.$$
		
		Note that when the measure is concentrated at the largest possible family size $\bm{u}$
		\begin{align*}
			P_u(D_n|\bm{Z_{n-1}})&=|\bm{u}|^{r_n}\binom{|\bm{Z_{n-1}}|}{r_n}\binom{|\bm{u}||\bm{Z_{n-1}}|}{r_n}^{-1}\\
			&=|\bm{u}|^{r_n}\prod_{a=0}^{r_n-1}\frac{|\bm{Z_{n-1}}|-a}{|\bm{u}||\bm{Z_{n-1}}|-a}\\
			&\geq |\bm{u}|^{r_n}\Biggl( \frac{|\bm{Z_{n-1}}|-r_n-1}{|\bm{u}||\bm{Z_{n-1}}|} \Biggr)^{r_n}\\
			&=\Biggl(1-\frac{r_n-1}{|\bm{Z_{n-1}}|}\Biggr)^{r_n}\\
			&\geq 1-\frac{r_n(r_n-1)}{|\bm{Z_{n-1}}|}\to 1 \text{ as } n \to \infty \\
		\end{align*}
		
		Where the almost sure convergence is due to Lemma \ref{lem:KSc}. By the dominated convergence theorem, $\P(D_n)=\E(\P(D_n|\bm{Z_{n-1}}))\geq \E(P_u(D_n|\bm{Z_{n-1}}))\rightarrow 1$ and the desired result follows for offspring distributions of bounded support.
		
		In the case of offspring distributions of possibly unbounded support, consider the event
		$G_n=\{|Y^{(i)}_{n-1,j}|\leq |\bm{Z_{n-1}}|; 1\leq i \leq l, 1\leq j\leq Z_{n-1,i} \}$ where all individuals produce less offspring than the total population size in their previous generation. From the Markov inequality, one gets
		
		\begin{align*}
			\P(G_n|\bm{Z_{n-1}})&=\prod_{i=1}^l\prod_{j=1}^{Z_{n-1,i}}\P(|Y^{(i)}_{n-1,j}|\leq |\bm{Z_{n-1}}||\bm{Z_{n-1}})\\
			&=\prod_{i=1}^l\P(|Y^{(i){n-1,1}}|\leq |\bm{Z_{n-1}}||\bm{Z_{n-1}})^{Z_{n-1,i}}\\
			&\geq \prod_{i=1}^l\Biggl( 1-\frac{\E(|Y^{(i){n-1,1}}|^2)}{|\bm{Z_{n-1}}|^2} \Biggr)^{Z_{n-1,i}}\\
			&\geq \Biggl( 1-\frac{C}{|\bm{Z_{n-1}}|^2} \Biggr)^{|\bm{Z_{n-1}}|}\\
		\end{align*}
		
		Where the last bound is due to assumption \ref{ass:cond} ii). Therefore,
		
		$$\P(G_n)\geq \E\Biggl(\biggl( 1-\frac{C}{|\bm{Z_{n-1}}|^2} \biggr)^{|\bm{Z_{n-1}}|} \Biggr)\geq \E\Biggl(1-\frac{C}{|\bm{Z_{n-1}}|}\Biggr).$$
		
		As in $G_n$ the maximum family size possible is $|\bm{Z_{n-1}}|$, the argument made in the case of compact supports yields that $\P(D_n|G_n)\geq \E\Bigl(1-\frac{r_n(r_n-1)}{|\bm{Z_{n-1}}|}\Bigr)$ and it follows that 
		
		\begin{align*}
		\P(D_n)&= \P(D_n|G_n)\P(G_n)-\P(D_n|G_n^c)\P(G_n^c)\\
			&\geq \P(D_n|G_n)\P(G_n)\\
			&\geq  \E\Biggl(1-\frac{r_n(r_n-1)}{|\bm{Z_{n-1}}|}\Biggr) -C\E(|\bm{Z_{n-1}}|^{-1})\rightarrow 0,\\
		\end{align*}
		where we appeal again to Lemma \ref{lem:KSc} to argue that the expected value of  $\frac{r_n(r_n-1)}{|\bm{Z_{n-1}}|}$ vanishes, which proves that $\P(D_n)\rightarrow 1$ as desired.
		
		Regarding the speed of convergence, let $\alpha$ as in the statement and note that
		
		\begin{align*}
			\rho^{\alpha n}r_n^{-2}(1-\P(D_n))&\leq\rho^{\alpha n}r_n^{-2} \E\Biggl(\frac{r_n(r_n-1)+C}{|\bm{Z_{n-1}}|}\Biggr)\\
			&\leq \rho^{\alpha n}\E(|\bm{Z_{n-1}}|^{-1})+\rho^{\alpha n}C\E(|\bm{Z_{n-1}}|^{-1})\\
			&\leq \rho^{\alpha n}\E(|\bm{Z_{n-1}}|^{-1})+\rho^{\alpha n}C\E(|\bm{Z_{n-1}}|^{-1})\\
			&\rightarrow 0\\
		\end{align*}
		
		With the last limit being a consequence of Lemma \ref{lem:aux}.
		
	\end{proof}
	
	\subsection{Proof of Theorem \ref{thm:iid}}
	
	\begin{proof}
		Denote, for integer $r$ as in the theorem statement
		
		$$\mathcal{F}_r=\{f:\{1, \ldots, r\} \rightarrow \{1, \ldots, l\} \}$$
		
		the set of all functions from $\{1, \ldots, r\}$ to $\{1, \ldots, l\}$. In context, $\mathcal{F}_r$ is the set of all possible parent type assignments in a sample of size $r$. For $f \in \mathcal{F}_r$, $i \in \{1, \ldots, l\}$ ,  $n \in \mathbb{N}$ and $U=\{u_1 ,\ldots , u_r\} \subset \mathbb{R}^l$ define also the total choices of offspring of type $i$ under assignment $f$ and observed vectors $U$ at generation $n$ as
		
		$$W_{f,i,n,U}=\binom{\sum_{j\in f^{-1}(\{i\})}|u_j|+\sum_{j=|f^{-1}(\{i\})|+1}^{Z_{n-1,i}}|Y_{n-1,j}^{(i)}|}{|f^{-1}(\{i\})|}$$
		
		The operator $||$ representing the $L^1$ norm when applied to a vector and the size of the set when applied to a pre image. To simplify notation, we'll use $W_{f,i,n,U}=W_{f,i}$ when the generation and the sample are understood. $W_{f,i}$ represents the number of choices of as many descendents of individuals of type $i$ as $f$ indicates out of a generation when the first $r$ members are known to be  $U=\{u_1 ,\ldots , u_r\}$ and nothing is known for the rest. An analogous combinatorial argument to that in the proof of Lemma \ref{lem:sam2} shows that
		
		\begin{multline}
			\P(X_1=u_1, \ldots, X_r=u_r,D_n)=  \sum_{\mathcal{F}_r}\Biggl(\prod_{w=1}^r p_{f(w)}(u_w)|u_w| \Biggr)\times \\  \E\Biggl(\prod_{s=0}^{r-1}(|\bm{Z_{n}}|-s)^{-1}\prod_{i=1}^l\Biggl[ \binom{Z_{n-1,i}}{|f^{-1}(\{i\})|}W_{f,i}^{-1}\prod_{t=0}^{|f^{-1}(\{i\})|-1}(S_{n-1,i}-t) \Biggr] \Biggr)
		\end{multline}

		Observe that, for fixed $t$ and fixed $f \in \mathcal{F}_r$, the ratio
		\begin{align*}
			&\Biggl|\frac{S_{n-1,i}-t}{\sum_{j\in f^{-1}(\{i\})}|u_j|+\sum_{j=|f^{-1}(\{i\})|+1}^{Z_{n-1,i}}|Y_{n-1,j}^{(i)}|-t}\Biggr|= \\
			& \Biggl|1-\frac{\sum_{j\in f^{-1}(\{i\})}|Y_{n-1,j}^{(i)}|-\sum_{j\in f^{-1}(\{i\})}|u_j|}{\sum_{j\in f^{-1}(\{i\})}|u_j|+\sum_{j=|f^{-1}(\{i\})|+1}^{Z_{n-1,i}}|Y_{n-1,j}^{(i)}| -t}\Biggr| \leq \\
			&2+\sum_{w=1}^r|u_w|
		\end{align*} is bounded and converges almost surely to one. For  fixed $t \in \{0 ,\ldots , |f^{-1}(\{i\})| \}$ and $s \in \{0 ,\ldots , r\}$
		
		$$\frac{Z_{n-1,i}-t}{|\bm{Z_{n}}|-s}\leq 1; \frac{Z_{n-1,i}-t}{|\bm{Z_{n}}|-s} \rightarrow \frac{b_i}{\rho} \mbox{ almost surely as } n \to \infty.$$
		
		With that, one has for any $f \in \mathcal{F}_r$
		
		$$\prod_{s=0}^{r-1}(|\bm{Z_{n}}|-s)^{-1}\prod_{i=1}^l\Biggl[ \binom{Z_{n-1,i}}{|f^{-1}(\{i\})|}W_{f,i}^{-1}\prod_{t=0}^{|f^{-1}(\{i\})|-1}(S_{n-1,i}-t) \Biggr] \leq (2+\sum_{w=1}^r|u_w|)^r $$
		
		$$\prod_{s=0}^{r-1}(|\bm{Z_{n}}|-s)^{-1}\prod_{i=1}^l\Biggl[ \binom{Z_{n-1,i}}{|f^{-1}(\{i\})|}W_{f,i}^{-1}\prod_{t=0}^{|f^{-1}(\{i\})|-1}(S_{n-1,i}-t) \Biggr] \rightarrow \rho^{-r}\prod_{i=1}^l b_i^{|f^{-1}(\{i\})|}$$
		
		So, as a consequence of the bounded convergence theorem, as $n\rightarrow  \infty$
		\begin{align*}
			\P(X_1=u_1, \ldots, X_r=u_r,D_n)&\rightarrow \rho^{-r}\sum_{\mathcal{F}_r}\Biggl(\prod_{w=1}^r p_{f(w)}(u_w)|u_w| \Biggr)\prod_{i=1}^l b_i^{|f^{-1}(\{i\})|}\\
			&=\rho^{-r}\sum_{\mathcal{F}_r}\Biggl(\prod_{w=1}^r p_{f(w)}(u_w)|u_w|b_{f(w)} \Biggr)\\
			&=\rho^{-r} (\prod_{w=1}^r |u_w|) \sum_{\mathcal{F}_r}\Biggl(\prod_{w=1}^r p_{f(w)}(u_w)b_{f(w)} \Biggr)\\
			&=\rho^{-r} (\prod_{w=1}^r |u_w|) \sum_{f(r)=1}^l \cdots \sum_{f(1)=1}^l \Biggl(\prod_{w=1}^r p_{f(w)}(u_w)b_{f(w)} \Biggr)\\
			&= \prod_{w=1}^r |u_w|\rho^{-1}\sum_{k=1}^l  p_{k}(u_w)b_{k} \\
			&=\prod_{w=1}^r p_S(u_w).\\
		\end{align*}
	
	By the triangle inequality.
		\begin{align*}
			&\Big|\P(X_w=u_w; w\in\{1, \ldots, r\})-\prod_{w=1}^rp_S(u_w)\Big|\\
			\leq & \Big|\P(X_w=u_w; w\in\{1, \ldots, r\})-\P(X_w=u_w; w\in\{1, \ldots, r\},D_n )\Big|\\
			+ &\Big|\P(X_w=u_w; w\in\{1, \ldots, r\},D_n )-\prod_{w=1}^rp_S(u_w)\Big|\\
			\leq & |\P(D_n^c )|+\Big|\P(X_w=u_w; w\in\{1, \ldots, r\},D_n )-\prod_{w=1}^rp_S(u_w)\Big| \to 0\\
		\end{align*}
		Where the term $|\P(X_w=u_w; w\in\{1, \ldots, r\},D_n )-\prod_{w=1}^rp_S(u_w)|$ vanishes by the preceeding result and $|\P(D_n^c )|$ vanishes by theorem \ref{thm:assympDn}. The desired result follows.
	\end{proof}
	
	\subsection{Proof of Theorem \ref{thm:asiid}}
	
	\begin{proof}
		
		With $\bm{X}=(X_1, \ldots, , X_{r_n})$ and $1(F)$ the indicator of the event $F$, denote
		
		$$\psi_n(\bm{t})=\E(1(D_n)\exp \{-i\langle\bm{t},\bm{X}\rangle \})$$
		
		It holds that
		
		\begin{align*}
			|\psi_n(\bm{t})-\phi_n(\bm{t})| & = |\E((1-1(D_n))\exp \{-i\langle\bm{t},\bm{X}\rangle \})| \\
			&\leq \E(|(1-1(D_n))\exp \{-i\langle\bm{t},\bm{X}\rangle \}|)\\
			& = \P(D_n^c). \\
		\end{align*}
		
		By assumption \ref{ass:cond}, theorem \ref{thm:assympDn} guarantees that $|\psi_n(\bm{t})-\phi_n(\bm{t})| \to 0$ uniformly on  $\bm{t}$. Using the notation $\bm{u} =(u_1, \ldots, u_{r_n})$ for $u_k \in\mathbb{N}^l$
		
		$$\phi(\bm{t})=\sum_{\bm{u} \in \mathbb{N}^{l\times r_n}}\left[ \frac{\Pi_{k=1}^{r_n}|u_k|}{\rho^{r_n}} \exp \{-i\langle\bm{t},\bm{u}\rangle \} \Biggl(\sum_{\mathcal{F}_{r_n}}\prod_{w=1}^{r_n} b_{f(w)} p_{f(w)}(u_w) \Biggr)\right]$$
		
		Note also that $|\psi_n(\bm{t})|\leq 1$ and
		
		\begin{align*}
			\psi_n(\bm{t})&=\E(1(D_n)\exp \{-i\langle\bm{t},\bm{X}\rangle \})\\
			&= \sum_{\bm{u} \in \mathbb{N}^{l\times r_n}}\P(\bm{X}=\bm{u}, D_n)\exp \{-i\langle\bm{t},\bm{u}\rangle\}\\
			&=\sum_{\bm{u} \in \mathbb{N}^{l\times r_n}}\sum_{\mathcal{F}_r}\Biggl(\prod_{w=1}^r p_{f(w)}(u_w)|u_w| \Biggr)\exp \{-i\langle\bm{t},\bm{u}\rangle\} \times \cdots \\
			& \E\Biggl(\prod_{s=0}^{r_n-1}(|\bm{Z_{n}}|-s)^{-1}\prod_{i=1}^l\Biggl[ \binom{Z_{n-1,i}}{|f^{-1}(\{i\})|}W_{f,i}^{-1}\prod_{t=0}^{|f^{-1}(\{i\})|-1}(S_{n-1,i}-t) \Biggr] \Biggr)\\
		\end{align*}
		
		So one has

		$$\phi(\bm{t})-\psi_n(\bm{t})=\sum_{\mathcal{F}_{r_n}}\sum_{\bm{u} \in \mathbb{N}^{l\times r_n}}\exp \{-i\langle\bm{t},\bm{u}\rangle\}A(\bm{u},f,n)V(\bm{u},f,n)$$

		Where 
		
		$$A(\bm{u},f,n)=\Biggl(\prod_{w=1}^{r_n} |u_w| p_{f(w)}(u_w) \rho^{-1} \Biggr)$$
		\begin{multline}\label{eq:a4}
			V(\bm{u},f,n)=\\ \prod_{w=1}^{r_n} b_{f(w)}-\E\Biggl(\prod_{s=0}^{r_n-1}\rho(|\bm{Z_{n}}|-s)^{-1}\prod_{i=1}^l\Biggl[ \binom{Z_{n-1,i}}{|f^{-1}(\{i\})|}W_{f,i}^{-1}\prod_{t=0}^{|f^{-1}(\{i\})|-1}(S_{n-1,i}-t)\Biggr] \Biggr).
		\end{multline}

		The proof of theorem \ref{thm:iid} shows that, for any $f \in \mathcal{F}_{r_n}$ and any $\bm{u} =(u_1 \ldots, u_{r_n})$ with $u_k \in\mathbb{N}^l$, as $n\rightarrow \infty$, $V(\bm{u},f,n)$ converges to zero, as the term inside the expected value in \eqref{eq:a4} is bounded and converges almost surely to $\prod_{w=1}^{r_n} b_{f(w)}=\prod_{w=1}^{r_n} b_{w}$. Moreover, for any $f \in \mathcal{F}_{r_n}$ it holds that
		
		$$|V(\bm{u},f,n)|\leq \E\Biggl( \prod_{i=1}^l\prod_{j=0}^{|f^{-1}\{i\}-1|} \Biggl|\frac{\sum_{ k \in f^{-1}\{i\}}(|u_k|-Y_{n-1,k}^{(i)})}{S_{n-1,i}+\sum_{ k \in f^{-1}\{i\}}(|u_k|-Y_{n-1,k}^{(i)})-j}\Biggr| \Biggr),$$ so a straightforward application of Jensen's inequality yields 
		
	    $$|V(\bm{u},f,n)|\leq \exp\Biggl(-\sum_{k=1}^{r_n}|u_k|+r_n\Biggr),$$which implies that for any value of $\bm{t}$,
		
		\begin{align*}
			|\phi(\bm{t})-\psi_n(\bm{t})|&\leq \sum_{\mathcal{F}_{r_n}}\sum_{\bm{u} \in \mathbb{N}^{l\times r_n}}A(\bm{u},f,n)|V(\bm{u},f,n)|\\
			&\leq \sum_{\mathcal{F}_{r_n}}\sum_{\bm{u} \in \mathbb{N}^{l\times r_n}}A(\bm{u},f,n)e^{-\sum_{k=1}^{r_n}|u_k|+r_n}\\
			&=\E\Biggl(\prod_{k=1}^{r_n} e^{-|B_i|+1} \Biggr)=\left[\E( e^{-|B_1|+1} )\right]^{r_n}\rightarrow 0,\\
		\end{align*} where $B_1 \ldots, B_{r_n}$ are $r_n$ iid $l$-dimensional random variables with probability mass function given by $p_S$, where convergence is due to the fact that $|B_1|\geq 1$ almost surely and $\P(|B_1|>1)$ together imply that $\E( e^{-|B_1|+1} )<1$.

	To conclude, note that for any $\bm{t}$
		
		$$|\phi(\bm{t})-\phi_n(\bm{t})|\leq |\phi_n(\bm{t})-\psi_n(\bm{t})|+|\phi(\bm{t})-\psi_n(\bm{t})|$$
		
		As both differences on the right hand side were shown to vanish asymptotically, taking limits as $n\rightarrow \infty$ yields the desired result.
		
	\end{proof}
	
	\subsection{Proof of Theorem \ref{thm:amle}}
	
	\begin{proof}
		Denote $$f_d(\theta)=\frac{\partial \log \prod_{j=1}^{r_n}p_S(X_j)}{\partial \theta_d}=\sum_{j=1}^{r_n}\frac{\partial \log p_S(X_j)}{\partial \theta_d}$$
		
		As $f'_d(\theta)=\frac{\partial f_d(\theta)}{\partial \theta_d}$ is continuous, the mean value theorem guarantees the existence of $\bar{\theta}$ between $\theta_{ML}$ and $\tilde{\theta}$ such that 
		
		\begin{equation} \label{eq:amle1}
			f'_d(\bar{\theta})[(\theta_{ML})_d-\tilde{\theta}_d]=f_d(\theta_{ML})-f_d(\tilde{\theta}).
		\end{equation}
		
		As $\theta_{ML}$ maximizes $\log\left(\prod_{j=1}^{r_n}p_S(X_j)\right)$, which is continuously differentiable, it follows that $f_d(\theta_{ML})=0$ and \eqref{eq:amle1} can be rewritten as 
		
		\begin{equation} \label{eq:amle2}
			\sqrt{r_n}[(\theta_{ML})_d-\tilde{\theta}_d]=-\frac{f_d(\tilde{\theta})/\sqrt{r_n}}{f'_d(\bar{\theta})/r_n}.
		\end{equation}
		
		As the domain of $p_1, \ldots, p_l$ does not depend on parameter $\theta$, from Fubini's Theorem and the fact that the sum of probability masses is always 1 follows that $\E\bigg(\frac{\partial \log p_S(X)}{\partial \theta_d}\bigg)=\sum_{\bm{u}: p_S(\bm{u})>0} \frac{\partial \log p_S(\bm{u})}{\partial \theta_d}p_S(\bm{u})=\sum_{\bm{u}: p_S(\bm{u})>0} \frac{\partial p_S(\bm{u})}{\partial \theta_d}=\frac{\partial \sum_{\bm{u}: p_S(\bm{u})>0}p_S(\bm{u})}{\partial \theta_d}=0$. The Central Limit Theorem then yields
		
		$$f_d(\tilde{\theta})/\sqrt{r_n} =  \frac{1}{\sqrt{r_n}}\sum_{j=1}^{r_n}\frac{\partial \log p_S(X_j)}{\partial \theta_d} \rightarrow_\mathcal{L} \mathcal{N}(0, s_d),$$with $s_d=\E\Biggl(\Biggl[\frac{\partial \log p_S(X)}{\partial \theta_d}\Biggr]^2\Biggr)=\sum_{\bm{u}: p_S(\bm{u})>0} \Biggl[\frac{\partial p_S(\bm{u})}{\partial \theta_d}\Biggr]^2\frac{1}{p_S(\bm{u})}$ the variance of $\frac{\partial \log p_S(X)}{\partial \theta_d}$. Similarly, the strong law of large numbers and the continuous mapping theorem, 
		
		$$\frac{f'_d(\bar{\theta})}{r_n} \to \sum_{\bm{u}: p_S(\bm{u})>0} \frac{\partial^2 \log p_S(\bm{u})}{\partial \theta^2_d}p_S(\bm{u})=\sum_{\bm{u}: p_S(\bm{u})>0}\frac{\partial^2 p_S(\bm{u})}{\partial \theta^2_d}+ \Biggl[\frac{\partial p_S(\bm{u})}{\partial \theta_d}\Biggr]^2\frac{1}{p_S(\bm{u})}=s_d$$
		
		Using Slutsky's Theorem, we conclude that $\sqrt{r_n}[(\theta_{ML})_d-\tilde{\theta}_d]\rightarrow_\mathcal{L} \mathcal{N}(0, s_d^{-1})$ as desired.
		
	\end{proof}
	
	\subsection{Proof of Theorem \ref{thm:moments}}
	
	\begin{proof}
		Let $\bm{X}\sim p_S$ be an $l$-dimensional random variable. One has $\E(|\bm{X}|^{-1})=\sum_{u\in \mathbb{R}^l}\sum_{i=1}^lb_ip_i(\bm{u})\rho^{-1}=\rho^{-1}$ and $ \E\left(\frac{X_j}{|\bm{X}|}\right)=\sum_{u\in \mathbb{R}^l}\sum_{i=1}^l u_jb_ip_i(\bm{u})\rho^{-1}=\rho^{-1}\sum_{i=1}^lM_{i,j}b_i=b_j$, so by the Central Limit Theorem it follows that
		
		$$\sqrt{r_n}(T_n-\rho^{-1})\rightarrow_\mathcal{L} \mathcal{N}_1(0,\sigma_T^2), \quad \sqrt{r_n}(U_n-b)\rightarrow_\mathcal{L} \mathcal{N}_l(\bm{0}, \Sigma),$$ the second convergence result being the desired one. Using the delta method with the function $f(x)=x^{-1}$ on the asymptotic distribution of $\sqrt{r_n}(T_n-\rho^{-1})$ proves that 
		
		$$\sqrt{r_n}(T_n^{-1}-\rho)\rightarrow_\mathcal{L} \mathcal{N}_1(0,\sigma_T^2\rho^{-4}),$$ and the theorem is proven.
		
	\end{proof}

\end{document}